\documentclass[
	paper=a4,
	DIV=13,
    ]{scrartcl}

\usepackage{amsmath,amsthm,amssymb}
\usepackage[utf8]{inputenc}
\usepackage{booktabs}
\usepackage{mathtools}
\usepackage{enumitem}
\usepackage[
	backend=biber,
	natbib=true,
	isbn=false,
	date=year,
	bibencoding=utf8,
]{biblatex}
\usepackage[colorlinks=true,linkcolor={blue},citecolor={blue},urlcolor={blue}]{hyperref} 
\hypersetup{pdfauthor={Jan Blechschmidt, Roland Herzog, Max Winkler},pdftitle={Error estimation for second-order PDEs in non-variational form}}

\usepackage{subcaption}
\usepackage[capitalize,
	noabbrev,
	nameinlink
]{cleveref}

\usepackage{LaTeX_others}
\usepackage{LaTeX_boldsymbols}
\usepackage{LaTeX_tikz}

\usepackage{todonotes}

\newtheorem{definition}{Definition}
\newtheorem{lemma}{Lemma}
\newtheorem{theorem}{Theorem}

\newtheorem{remark}{Remark}

\definecolor{darkgreen}{rgb}{0,0.6,0}
\definecolor{darkblue}{rgb}{0,0,0.6}
\definecolor{darkred}{rgb}{0.8,0.0,0.0}

\usepackage{pgfplots}
\usepackage{pgfplotstable}

\usepackage{siunitx} 
\tikzset{fs/.style={#1,solid,mark size=.9pt,thick,->},
	fs/.default={black},%
	cgZeroStab/.style={blue, solid, mark=+, mark options={solid, blue}},
	cgOneStab/.style={cyan, solid, mark=*, mark options={solid, cyan}},
	cgTwoStab/.style={teal, solid, mark=diamond*, mark options={solid, teal}},
	dgZeroStab/.style={orange, solid, mark=pentagon*, mark options={solid, orange}},
	dgOneStab/.style={red, solid, mark=triangle*, mark options={solid, red}},
	Neilan/.style={magenta, solid, mark=square*, mark options={solid, magenta}},
	NeilanSalgadoZhang/.style={darkgray, solid, mark=star, mark options={solid, darkgray}},
	uniformL2/.style={blue, dashed, mark=diamond*, mark options={solid, blue}},
	adaptiveL2/.style={blue, solid, mark=diamond*, mark options={solid, blue}},
	uniformH1/.style={cyan, dashed, mark=square, mark options={solid, cyan}},
	adaptiveH1/.style={cyan, solid, mark=square, mark options={solid, cyan}},
	uniformH2/.style={teal, dashed, mark=triangle, mark options={solid, teal}},
	adaptiveH2/.style={teal, solid, mark=triangle, mark options={solid, teal}}  
}

\tikzset{myStyle/.style={#1,thick,->},
	myStyle/.default={red},%
	myS1/.style={myStyle},
	myS2/.style={myStyle=blue},
	L2uniform/.style={red, solid, mark=+,mark size=.9pt},
	L2adaptive/.style={red, dashed, mark=+,mark size=.9pt},
	H1uniform/.style={blue, solid, mark=+,mark size=.9pt},
	H1adaptive/.style={blue, dashed, mark=+,mark size=.9pt},
	H2uniform/.style={darkgreen, solid, mark=+,mark size=.9pt},
	H2adaptive/.style={darkgreen, dashed, mark=+,mark size=.9pt},
	H2huniform/.style={orange, solid, mark=+,mark size=.9pt},
	H2hadaptive/.style={orange, dashed, mark=+,mark size=.9pt},
	Etauniform/.style={magenta, solid, mark=o,mark size=.9pt},
	Etaadaptive/.style={magenta, dashed, mark=o,mark size=.9pt},
	pw/.style={red, solid, mark=*, mark size=1.3pt},
	fe0/.style={blue, solid, mark=square*, mark size=1.3pt},
	fe1/.style={darkgreen, solid, mark=triangle*, mark size=1.3pt},
	fe2/.style={magenta, solid, mark=diamond*, mark size=1.3pt},
}

\newcommand{\eg}{e.g.}
\newcommand{\ie}{i.e.}
\newcommand{\Cordes}{Cordes}
\newcommand{\Ctr}{C_{\textup{tr}}}
\newcommand{\Cv}{C_\textup{v}}
\newcommand{\CE}{C_{\mathbb E}}
\newcommand{\CH}{C_{\FEHessian{}}}

\newcommand{\uconf}{\widetilde{u_h}}

\newcommand{\uSpace}{\mathcal V_{h}}
\newcommand{\uSpaceZero}{\mathcal V_{h,0}}
\newcommand{\uSpaceConf}{\mathcal V_{h,\text{conf}}}
\newcommand{\HSpace}{\mathcal W_{h}}
\newcommand{\HSpaceCG}{\mathcal W_h^{\textup{CG}}}
\newcommand{\HSpaceDG}{\mathcal W_h^{\textup{DG}}}

\newcommand{\MyDGSpace}{\HSpaceDG{^{(p-2)}}}

\newcommand{\Hnorm}[1]{\norm{#1}_{H^2_h(\Omega)}}

\renewcommand\restr[2]{{
	\left.\kern-\nulldelimiterspace 
		#1 
		\vphantom{|} 
		\right|_{#2} 
}}

\newcommand{\Hessian}{\nabla^2}
\newcommand{\FELaplace}[1]{\ifthenelse{\equal{#1}{}}{\mathbb H_\Delta}{\mathbb H_{\Delta}(#1)}}
\newcommand{\FELaplaceCG}[1]{\ifthenelse{\equal{#1}{}}{\mathbb H^{\textup{CG}}_\Delta}{\mathbb H^{\textup{CG}}_{\Delta}(#1)}}
\newcommand{\FELaplaceDG}[1]{\ifthenelse{\equal{#1}{}}{\mathbb H^{\textup{DG}}_\Delta}{\mathbb H^{\textup{DG}}_{\Delta}(#1)}}
\newcommand{\FEHessian}[1]{\ifthenelse{\equal{#1}{}}{\mathbb H}{\mathbb H(#1)}}
\newcommand{\FEHessianCG}[1]{\ifthenelse{\equal{#1}{}}{\mathbb H^{\textup{CG}}}{\mathbb H^{\textup{CG}}(#1)}}
\newcommand{\FEHessianDG}[1]{\ifthenelse{\equal{#1}{}}{\mathbb H^{\textup{DG}}}{\mathbb H^{\textup{DG}}(#1)}}
\newcommand{\hij}[1]{\ifthenelse{\equal{#1}{}}{\mathbb H_{ij}}{\mathbb H_{ij}(#1)}}
\newcommand{\hii}[1]{\FEHessian{}_{ii}(#1)}

\newcommand{\LiftPP}{\mathbb E_{\HSpaceCG}} 
\newcommand{\LiftHCT}{\mathbb E_{h}}

\newcommand{\I}{\mathbb{I}}
\newcommand{\N}{\mathbb{N}}
\renewcommand{\P}{\mathbb{P}}
\newcommand{\R}{\mathbb{R}}
\newcommand{\W}{\mathbb{W}}
\newcommand{\FF}{\mathcal{F}}
\newcommand{\LL}{\mathcal{L}}

\newcommand{\PP}{\mathcal{P}}
\newcommand{\TT}{\mathcal{T}}

\newcommand{\Edges}{\FF_h}
\newcommand{\InnerEdges}{\FF_h^{\mathcal{I}}}

\newcommand{\constCordes}{\varepsilon}
\newcommand{\constEllipticity}{\lambda_\text{E}}

\newcommand{\email}[1]{\href{mailto:#1}{#1}}

\title{Error estimation for second-order PDEs in non-variational form}
\author{
		Jan Blechschmidt%
		\thanks{Technische Universität Chemnitz, Faculty of Mathematics, Professorship Numerical Mathematics (Partial Differential Equations), 09107 Chemnitz, Germany, \email{jan.blechschmidt@mathematik.tu-chemnitz.de}, \url{https://www.tu-chemnitz.de/mathematik/part_dgl/people/blechschmidt}}
		\and
		Roland Herzog%
		\thanks{Technische Universität Chemnitz, Faculty of Mathematics, Professorship Numerical Mathematics (Partial Differential Equations), 09107 Chemnitz, Germany, \email{roland.herzog@mathematik.tu-chemnitz.de}, \url{https://www.tu-chemnitz.de/mathematik/part_dgl/people/herzog}}
		\and
		Max Winkler%
		\thanks{Technische Universität Chemnitz, Faculty of Mathematics, Professorship Numerical Mathematics (Partial Differential Equations), 09107 Chemnitz, Germany, \email{max.winkler@mathematik.tu-chemnitz.de}, \url{https://www.tu-chemnitz.de/mathematik/part_dgl/people/winkler}}%
}

\begin{document}
\pgfplotstableset{
	multicolumn names, 
	col sep=comma, 
	sci zerofill,
	columns/Ndofs/.style={
		column name=$\dim(V_h)$, 
	int detect},  
	columns/NdofsMixed/.style={
		column name=$\dim(V_h + W_h^{2 \times 2})$,
	int detect},
	columns/L2mixed/.style={string type,
		column type=l,
	column name=$\norm{u-u_h}_{L^2(\Omega)}$},
	create on use/L2mixed/.style={
		create col/assign/.code={%
			\edef\entry{
				\noexpand\pgfmathprintnumber[sci e, fixed zerofill, precision=2]{
					\thisrow{L2_error}
				}
				\ifnum\pgfplotstablerow=0
				\else
					(\noexpand\pgfmathprintnumber[fixed zerofill, precision=2]{\thisrow{L2_eoc}})
				\fi
			}%
			\pgfkeyslet{/pgfplots/table/create col/next content}\entry
		}
	},
	columns/H1mixed/.style={string type,
		column type=l,
	column name=$\norm{u-u_h}_{H^1_0(\Omega)}$},
	create on use/H1mixed/.style={
		create col/assign/.code={%
			\edef\entry{
				\noexpand\pgfmathprintnumber[sci, sci e, fixed zerofill, precision=2]{
					\thisrow{H1_error}
				}
				\ifnum\pgfplotstablerow=0
				\else
					(\noexpand\pgfmathprintnumber[fixed zerofill, precision=2]{\thisrow{H1_eoc}})
				\fi
			}%
			\pgfkeyslet{/pgfplots/table/create col/next content}\entry
		}
	},
	columns/H2mixed/.style={string type,
		column type=l,
	column name=$\norm{u-u_h}_{H^2_0(\Omega)}$},
	create on use/H2mixed/.style={
		create col/assign/.code={%
			\edef\entry{
				\noexpand\pgfmathprintnumber[sci, sci e, fixed zerofill, precision=2]{
					\thisrow{H2_error}
				}
				\ifnum\pgfplotstablerow=0
				\else
					(\noexpand\pgfmathprintnumber[fixed zerofill, precision=2]{\thisrow{H2_eoc}})
				\fi
			}%
			\pgfkeyslet{/pgfplots/table/create col/next content}\entry
		}
	},
	columns/H2hmixed/.style={string type,
		column type=l,
	column name=$\norm{u-u_h}_{H^2_h(\Omega)}$},
	create on use/H2hmixed/.style={
		create col/assign/.code={%
			\edef\entry{
				\noexpand\pgfmathprintnumber[sci, sci e, sci zerofill, precision=2]{
					\thisrow{H2h_error}
				}
				\ifnum\pgfplotstablerow=0
				\else
					(\noexpand\pgfmathprintnumber[fixed, fixed zerofill, precision=2]{\thisrow{H2h_eoc}})
				\fi
			}%
			\pgfkeyslet{/pgfplots/table/create col/next content}\entry
		}
	},
	columns/L2_error/.style={
		column name=$\norm{u-u_h}_{L^2(\Omega)}$,
		sci,
		sci e,
	sci zerofill},
	columns/L2_eoc/.style={
		string replace={0}{},
		column name=$\text{EOC}_{L^2}$,
		precision=2,
		fixed,
		fixed zerofill,
		postproc cell content/.append style={/pgfplots/table/@cell content/.add={(}{)},},
	},
	columns/H1_error/.style={
		column name=$\norm{u-u_h}_{H^1_0(\Omega)}$,
		sci,
		sci e,
	sci zerofill},
	columns/H1_eoc/.style={
		string replace={0}{},
		column name=$\text{EOC}_{H^1_0}$,
		sci,
		sci zerofill, 
	precision=2, fixed},
	columns/H2_error/.style={
		column name=$\norm{u-u_h}_{H^2_0(\Omega)}$,
		sci,
		sci e,
	sci zerofill},
	columns/H2_eoc/.style={
		string replace={0}{},
		column name=$\text{EOC}_{H^2_0}$,
		precision=2,
		fixed,
	fixed zerofill},
	columns/H2h_error/.style={
		column name=$\norm{u-u_h}_{H^2_h(\Omega)}$,
		sci,
		sci e,
	sci zerofill},
	columns/H2h_eoc/.style={
		string replace={0}{},
		column name=$\text{EOC}_{H^2_h}$,
		sci,
		sci zerofill, 
	precision=2, fixed},
	columns/gmres_iter/.style={
		column name=$N_\text{gmres}$,
	column type={S},string type},
	every head row/.style={
		before row={\toprule}, 
		after row={\midrule} 
	},
	every last row/.style={after row=\bottomrule}, 
}

\maketitle

\begin{abstract}
	Second-order partial differential equations in non-divergence form are considered.
	Equations of this kind typically arise as subproblems for the solution of Hamilton-Jacobi-Bellman equations in the context of stochastic optimal control, or as the linearization of fully nonlinear second-order PDEs. 
	The non-divergence form in these problems is natural.
	If the coefficients of the diffusion matrix are not differentiable, the problem can not be transformed into the more convenient variational form.

	We investigate tailored non-conforming finite element approximations of second-order PDEs in non-divergence form, utilizing finite-element Hessian recovery strategies to approximate second derivatives in the equation. 
	We study both approximations with continuous and discontinuous trial functions.
	Of particular interest are a~priori and a~posteriori error estimates as well as adaptive finite element methods.
	In numerical experiments our method is compared with other approaches known from the literature.
\end{abstract}

\section{Introduction}
\label{sec:Introduction}

Many boundary value problems feature linear, second-order partial differential equations in divergence form. That is, the differential operator may be written as
\begin{equation}\label{eq:operator_in_divergence_form}
	\widetilde \LL u \coloneqq \div(\widetilde A\,\nabla u) + \widetilde b^{\,\top}\,\nabla u + \widetilde c\,u
\end{equation}
with coefficients $\widetilde A\colon\Omega\to\R^{d\times d}$,
$\widetilde b\colon\Omega\to\R^d$, $\widetilde c\colon \Omega\to\R$. Here and in the following
$\Omega\subset\R^d$, $d\in\N$, is a bounded domain.
Although formulation \eqref{eq:operator_in_divergence_form} covers a wide range of applications, there are some linear problems involving operators in non-divergence form
\begin{equation}
	\label{eq:general_operator}
	\LL u \coloneqq A\dprod\nabla^2u + b^\top \,\nabla u + c\,u.
\end{equation}
Here, $A \dprod B$ denotes the Frobenius inner product
$\sum_{i,j=1}^d a_{ij}\,b_{ij}$ of two matrices $A, B \in \R^{d\times d}$,
and $A\colon \Omega\to\R^{d\times d}$, $b\colon \Omega\to\R^d$ and $c\colon \Omega\to \R$ are given coefficients.
The matrix~$A$ is assumed to be almost everywhere positive definite and symmetric.

Classical and strong solutions of problems in non-divergence form with Hölder-regular or continuous coefficients, respectively, have been analyzed in \cite[Ch.~6, 9]{GilbargTrudinger2001}.
In the applications of interest here, however, coefficients are only bounded and measurable.
Under even higher smoothness assumptions on the coefficient~$A$,
a non-divergence form operator \eqref{eq:general_operator} can be transformed into an operator in divergence form \eqref{eq:operator_in_divergence_form}
with $\widetilde A = A$ and $\widetilde b=b-\Div A$, where $\Div A$ denotes the row-wise
divergence of the matrix~$A$.
Even if $A$ is smooth, however, this transformation may lead to convection dominated problems which
induce further challenges.

Our aim in this paper is to investigate the boundary value problem
\begin{equation}
	\label{eq:bvp}
	\begin{aligned}
		\LL u & = f & &\text{in } \Omega,
		\\
		u & = 0 & & \text{on } \Gamma \coloneqq \partial\Omega,
	\end{aligned}
\end{equation}
for some source term $f\in L^2(\Omega)$.
Let us briefly mention some applications where problems of this kind are of interest.
Naturally, linear problems with operators in non-divergence form arise in the context of stochastic differential equations due to the It\^o formula, see
\cite{BarlesJakobsen2002,BlechschmidtHerzog2018,FlemingRishel1975,Pham2009}.
Such problems play a central role in financial mathematics, \eg, the valuation of financial products.
A closely connected area is the numerical solution of second-order Hamilton-Jacobi-Bellman (HJB) equations \cite{SmearsSueli2014,BrennerKawecki2019_preprint}, where the existence of an operator in non-divergence form also follows due to the stochastic influence.
In addition to the non-variational nature of the linear operator, these problems possess further numerical challenges due to nonlinearities introduced by a pointwise minimization.
A further application is the solution of highly nonlinear second order
partial differential equations. A linearization used, \eg, in a Newton method,
leads to a problem of the form \eqref{eq:bvp} in the general case.
Typical examples include the Monge-Amp\`ere equation \cite{BenamouFroeseOberman2010,BrennerNeilan2012,DeanGlowinski2006,FengNeilan2009,Gutierrez2001,Neilan2014,TrudingerWang2008} which reads
$\det(\nabla^2 u) = (\partial_{xx} u) \, (\partial_{yy} u) - (\partial_{xy} u)^2= f$ in case of $d=2$.
The linearization at a function $u_0$ leads to a differential operator of the form
\eqref{eq:general_operator} with
\begin{equation*}
	A = \begin{pmatrix}
		\phantom{-}\partial_{yy} u_0 & -\partial_{xy} u_0 \\
		-\partial_{xy} u_0 & \phantom{-}\partial_{xx} u_0
	\end{pmatrix}.
\end{equation*}

For the solution of problems in the form \eqref{eq:bvp} several different approaches avoiding
the transformation into a divergence form PDE have recently been studied.
Many approaches aim at approximating strong solutions, \ie,
\begin{equation}
	\label{eq:strong_form}
	\text{Find }\ u\in H^2(\Omega)\cap H_0^1(\Omega):\quad \int_\Omega \LL u\,v\,\dx = \int_\Omega f\,v\,\dx \quad \forall v \in L^2(\Omega).
\end{equation}
A discrete approximation of the solution $u\in H^2(\Omega)\cap H_0^1(\Omega)$ of \eqref{eq:strong_form}
in the case $\LL=A\dprod\nabla^2$ is usually obtained by solving a problem of the form
\begin{equation}
	\label{eq:mixed_form_general}
	\text{Find }\ u_h\in \uSpaceZero:\quad \int_\Omega A\dprod\FEHessian{u_h}\,\tau_h(v_h)\,\dx = \int_\Omega f\,\tau_h(v_h)\,\dx \quad \forall v_h\in \uSpaceZero,
\end{equation}
where $\uSpaceZero \coloneqq \uSpace\cap H_0^1(\Omega)$ is a finite-dimensional trial and test space
and $\FEHessian{u_h}$ is an approximation of the
Hessian $\nabla^2 u$, also sought in a finite-dimensional space $\HSpace(\R^{d\times d})$ with discretization parameter $h>0$.
Several approaches have been studied in the literature and most discretization strategies differ in the choice of the discrete spaces $\uSpace$, $\HSpace$, the approximation $\FEHessian{}$ of the Hessian and the realization of the test function $\tau_h\colon \uSpace\to L^2(\Omega)$. 

Let us briefly summarize the most prominent approaches.
The first article discussing a direct treatment of a non-variational problem,
to the best of the authors' knowledge, is \citet{LakkisPryer2011}. Therein,
$\uSpace$ and $\HSpace$ consist of continuous Lagrange finite element functions of order $p\ge 1$,
the choice $\tau_h = \text{id}$ is used 
and the finite-element Hessian $\FEHessian{u_h}\in \HSpace(\R^{d\times d})$ is obtained by a discrete version of the integration-by-parts formula, \ie,
\begin{equation}
	\label{eq:def_discrete_Hessian}
	\int_\Omega \FEHessian{u_h}\dprod v_h\,\dx = -\int_\Omega \nabla u_h \cdot \Div v_h\,\dx
	+ \int_\Gamma \nabla u_h \cdot \, (v_h\,n_\Gamma)\,\ds \quad \forall v_h\in \HSpace(\R^{d\times d})
	.
\end{equation}
Here, $n_\Gamma\colon \Gamma\to \R^d$ denotes the outer normal vector on
$\Gamma\coloneqq\partial\Omega$.
A closely related approach using 
a discontinuous Galerkin approximation for the finite-element Hessian $\FEHessian{u_h}$ is studied
by~\citet{Neilan2017}.

There are other approaches that avoid the coupling with an additional variational formulation
used for the computation of a Hessian approximation. This
is possible when using the cell-wise exact Hessian $\FEHessian{}\coloneqq\nabla_h^2$
but additional jump penalty terms over the interior cell edges/faces have to be added to the
bilinear form.
This idea is first studied by \citet{SmearsSueli2013}, under the weak assumption that $A$ belongs to $L^\infty(\Omega;\R^{d\times d})$ and fulfills a so-called \Cordes\ condition.
In that work, discontinuous Galerkin approximations and the choice $\tau_h(v_h) = \Delta_h v_h$ are used and appropriate jump penalty terms are added to the bilinear form so that discrete coercivity is guaranteed.
Quite similar is the approach of \citet{NeilanSalgadoZhang2017} who use continuous Lagrange elements. 
In both approaches the coercivity is shown via a discrete Miranda-Talenti estimate.
In a related line of research, \citet{FengHenningsNeilan2017} employ continuous Lagrange finite elements using the choice $\tau_h=\id$.
They show well-posedness of the discrete scheme ensuing via a discrete inf-sup condition.
For this approach, at least continuity of the coefficients of $A$ has to be assumed as a localization argument by freezing the coefficients of $A$ is applied in the proofs.
Analogous results are presented in \citet{FengNeilanSchnake2018} for a discontinuous Galerkin approximation.
Finally, an extension of the technique of \citet{SmearsSueli2013} to curved domains can be found in \cite{Kawecki2019}.

Before continuing, it is worth pointing out that the respective discrete linear systems and the techniques employed to prove their well-posedness differ in the references above and have far-reaching implications for computational practice.
	In particular, \cite{FengHenningsNeilan2017,FengNeilanSchnake2018} rely on a discrete Calderon-Zygmund estimate and therefore on a continuity assumption for the leading coefficients of the differential operator, as well as sufficiently fine meshes.
	Unfortunately, the former is typically not satisfied for HJB equations, which we have in mind as future applications.
	Moreover, the requirement of sufficiently fine initial meshes obstructs the utility of an adaptive mesh refinement strategy, which we develop here.
	Such limitations are not present in discretization approaches relying on the Cordes condition, including \cite{SmearsSueli2013,NeilanSalgadoZhang2017,Kawecki2019} and the present work.

A further method, which is proposed by \citet{Gallistl2017}, is based on a stabilized mixed finite element discretization involving an approximation of the gradient $w_h\in \HSpace(\R^d)$ by $\int_\Omega(\nabla u_h - w_h)\cdot\nabla v_h \, \dx = 0$ for all $v_h\in \uSpace$.  
This is, to the best of our knowledge, the first contribution proving also a~posteriori error estimates and the convergence of an adaptive finite element
method for the solution of non-divergence form PDEs.

In order to complete our survey, we want to mention that there are many further approaches
that do not directly fit into the framework \eqref{eq:mixed_form_general}.
This includes for instance regularization approaches like the \emph{vanishing moment method} 
studied in \cite{FengNeilan2008} and the references therein,
the primal-dual weak Galerkin method \cite{WangWang2017},
the Alexandroff-Bakelman-Pucci (ABP) method \cite{NochettoZhang2018},
or certain finite element schemes based on a very weak formulation of the model problem \cite{Fuehrer2019_preprint}.

In the present paper we discuss a new method combining multiple ideas of the
previously outlined approaches. To be more precise, we consider the discrete
formulation \eqref{eq:mixed_form_general} with a finite-element Hessian obtained either by
continuous finite elements as in \eqref{eq:def_discrete_Hessian}
or by a discontinuous ansatz that we specify later.
For the test functions we use $\tau_h(u_h)= \FELaplace{u_h} \coloneqq \trace \FEHessian{u_h}$.

The main results of this article include a rigorous proof of the well-posedness of the discrete scheme,
which is also based on a discrete Miranda-Talenti estimate following from a \Cordes\ condition.
Moreover, preconditioning strategies for the resulting system of linear equations 
are studied and we observe in experiments that the preconditioner is robust with respect to the mesh parameter.
Furthermore, we study a~priori and reliable a~posteriori error estimates in the energy norm.
Based on the a~posteriori error estimates we implement an adaptive finite element method
and confirm by experiments that the convergence rate is optimal in all test
cases, even for less smooth solutions.

Our method combines several advantages of the previously mentioned approaches. First, it is applicable to problems with discontinuous coefficients $A$ and hence allows an extension to HJB equations.
Among the approaches presented in our survey, only \cite{SmearsSueli2013} and \cite{NeilanSalgadoZhang2017} possess this property as well.
Second, under additional assumptions, our discretization can be realized without the addition of stabilization terms, which would involve jump penalties at the cell interfaces.
In numerical experiments we observed that all approaches which do use stabilization terms do not converge with an optimal rate in the $L^2(\Omega)$-norm. This surprising observation deserves further investigation. In addition to our approach, only the methods from \cite{Neilan2017} and \cite{LakkisPryer2013} likewise exhibit optimal $L^2(\Omega)$ rates. It should be noted that the computational cost for the approaches using a Hessian recovery strategy, including the proposed scheme, is naturally higher than the cost for schemes relying on the broken exact Hessian. However, the advantages mentioned above may justify this additional effort.

\section{The continuous problem}
\label{sec:continuous_problem}

Throughout this article $\Omega \subset \R^d$, $d\in\N$, is a bounded and convex domain.
We consider the boundary value problem with a second-order differential operator in non-divergence form
\begin{equation}
	\label{eq:MainProblem}
	\begin{aligned}
		A \dprod \nabla^2 u &= f & & \text{in } \Omega,\\
		u &= 0 & & \text{on } \Gamma
	\end{aligned}
\end{equation}
with $f \in L^2(\Omega)$.
The coefficient matrix~$A$ is assumed to belong to 
$L^{\infty}(\Omega; \R^{d\times d})$, to be symmetric and 
uniformly positive definite, \ie, there exists a constant $\constEllipticity > 0$ such that
\begin{equation}\label{eq:ellipticity}
	\xi^T \, A\, \xi \ge \constEllipticity \, \abs{\xi}^2 \quad \text{for all } \xi \in \R^d,
\end{equation}
almost everywhere in $\Omega$.

As the coefficient matrix~$A$ is not necessarily differentiable, one can at most ask for 
\emph{strong solutions} of \eqref{eq:MainProblem}, \ie, functions 
$u\in X\coloneqq H^2(\Omega) \cap H_0^1(\Omega)$ solving
\begin{equation}\label{eq:strong_form_L}
	\int_\Omega A\dprod\nabla^2 u\,v\,\dx = \int_\Omega f\,v\,\dx \quad\forall v\in L^2(\Omega)
	.
\end{equation}
Since the Laplacian $\Delta\colon X \to L^2(\Omega)$
is bijective due to the convexity of $\Omega\subset \R^d$,
the latter  equation is equivalent to
\begin{equation}\label{eq:strong_form_laplace}
	\int_\Omega A\dprod\nabla^2 u\,\Delta v\,\dx = \int_\Omega f\,\Delta v\,\dx \quad\forall v\in X.
\end{equation}

Existence of strong solutions follow for instance under the slightly stronger assumption $A\in C(\overline\Omega; \R^{d\times d})$ and when $\Gamma\coloneqq \partial\Omega$ is of class $C^{1,1}$, even for non-convex domains, see \cite[Theorem~9.15]{GilbargTrudinger2001}.

Another idea, which implies well-posedness even for general convex domains 
and which allows for discontinuous coefficients, is to impose a \Cordes\ condition,
\ie, the existence of a constant $\constCordes \in (0,1]$ such that
\begin{equation}
	\label{eq:CordesCondition}
	\frac{
		\norm{A}_F^2
	}{
		\bigh(){\trace A}^2
	}
	= 
	\frac{
		\sum_{i,j = 1}^d \normalh(){A_{ij}}^2
	}{
		\bigh(){\sum_{i=1}^d A_{ii}}^2
	}
	\le
	\frac{1}{d-1 + \constCordes}
	\quad \text{a.e.\ in } \Omega.
\end{equation}
In the two-dimensional case, this assumption follows from \eqref{eq:ellipticity}. 
As has been discussed in the recent literature, \eg\ \cite{SmearsSueli2013},
a rescaling of the equation \eqref{eq:strong_form_laplace} with the normalization coefficient
\begin{equation*}
	\gamma \coloneqq 
	\frac{
		\sum_{i=1}^d A_{ii}
	}{
		\sum_{i,j = 1}^d \bigh(){A_{ij}}^2
	} \in L^{\infty}(\Omega)
\end{equation*}
becomes advantageous in the analysis of the problem.
This can be explained with the following result, whose proof is stated in \cite[Lemma~1]{SmearsSueli2013}.
\begin{lemma}
	\label{lem:consequence_cordes}
	Assume that $A$ belongs to $L^\infty(\Omega;\R^{d\times d})$ and
	satisfies \eqref{eq:CordesCondition}.
	Then the inequality
	\begin{equation*}
		\norm{\gamma A - I}_F \le \sqrt{1-\varepsilon} \quad \text{a.e.\ in}\ \Omega
	\end{equation*}
	holds.
\end{lemma}
Obviously, \eqref{eq:CordesCondition} guarantees that the rescaled matrix $\gamma A$ is close to
the identity matrix, and consequently, the differential operator $\gamma A\dprod\nabla^2$ is
close to the elliptic Laplace operator.
Thus, if the \Cordes\ condition is fulfilled one can consider instead of
\eqref{eq:strong_form_laplace} a variational problem with the bilinear form
$a\colon X\times X\to\R$ defined by 
\begin{equation*}
	a(u,v) \coloneqq \int_\Omega \gamma A\dprod \nabla^2 u\, \Delta v\,\dx,
\end{equation*}
and the linear form $F\in X'$ defined by
\begin{equation*}
	F(v) \coloneqq \int_\Omega \gamma f\,\Delta v\,\dx.
\end{equation*}
The variational problem we are going to study in this article is defined by
\begin{equation}
	\label{eq:weak_form}
	\text{Find}\ u\in X \text{ such that }
	a(u,v) = F(v)\quad \forall v\in X.
\end{equation}
Under the assumption \eqref{eq:CordesCondition} the bilinear form~$a$ is
elliptic in $X=H^2(\Omega)\cap H_0^1(\Omega)$ and with the Lax-Milgram Lemma one can immediately prove the following result.
\begin{lemma}\label{lem:existence}
	Assume that the coefficient matrix $A$ belongs to $L^\infty(\Omega;\R^{d\times d})$ and fulfills the \Cordes\ condition \eqref{eq:CordesCondition} with $\constCordes\in (0,1]$. Then, the problem
	\eqref{eq:weak_form} possesses a unique solution $u\in H^2(\Omega)\cap H_0^1(\Omega)$.
	Moreover, the \emph{a~priori} estimate
	\begin{equation*}
		\norm{u}_{H^2(\Omega)} \le C_a \,\norm{f}_{L^2(\Omega)}
	\end{equation*}
	is fulfilled with some constant $C_a=C_a(d,\diam(\Omega),\lambda_E,\norm{A}_{L^\infty(\Omega)},\varepsilon)$.
\end{lemma}
\begin{proof}
	See \cite[Theorem~3]{SmearsSueli2013}.
\end{proof}

Notice that we restrict our discussion to problem \eqref{eq:MainProblem} mainly for notational simplicity.
For related investigations of equations involving also drift and potential terms, \ie,
the differential operator of the PDE is of the form \eqref{eq:general_operator} with $b\not\equiv 0$ and/or $c\not\equiv 0$, we refer the reader, \eg, to~\cite{SmearsSueli2014}.

\section{Discretization}
\label{sec:A-priori_Analysis}

We decompose our domain $\Omega$ into a family of feasible triangulations $\TT_h$ (triangular for $d=2$, tetrahedral for $d=3$) with discretization parameter $h=\max_{T\in\TT_h} h_T$, $h_T\coloneqq\diam(T)$. The diameter of the largest inscribed ball in a cell $T\in \TT_h$ is denoted by $\rho_T$. Throughout this article we assume that $\{\TT_h\}_{h>0}$ is shape-regular, \ie, there holds
\begin{equation*}
	\kappa_T\coloneqq\frac{h_T}{\rho_T} \le \kappa\quad \forall T\in \TT_h,
\end{equation*}
where the maximal aspect ratio~$\kappa$ is independent of $h$.
Moreover, meshes are considered which have a limited variation in the element size of neighboring elements, \ie, there is a constant $\Cv>0$ with $h_{T'} \le \Cv\,h_T$ for each $T,T'\in\TT_h$, $T\cap T'\ne\emptyset$.

By $\Edges$ we denote the set of facets of $\TT_h$ and by $n_F$ a unit normal
vector on $F\in\Edges$.
The normal vectors $n_F$ are chosen to point outwards if $F$ is a boundary facet and it has arbitrary but fixed orientation for interior facets. The diameter of a facet $F\in \Edges$ is denoted by $h_F$. Moreover, we denote the set of facets in the interior by $\InnerEdges$.
This includes all facets in the intersection of two elements in $\TT_h$.
Entities on either side of an interior facet are denoted by $\cdot^+$ and
$\cdot^-$, respectively, chosen in such a way that for $F=\partial T^+\cap
\partial T^-$, $n_F$ points towards $T^+$.

By $\PP_p(T)$, $T\in\TT_h$, we denote the set of polynomials on $T$ of degree not larger than $p\in \N$. Throughout this article generic constants are denoted by $c_{v_1,v_2,\ldots}$ where $v_1,v_2,\ldots$, are the quantities they depend on.

Furthermore, we introduce the following average and jump operators.
The average operators are defined by
\begin{equation*}
	\restr{\avg{u}}{F} \coloneqq \frac12(u^+ + u^-) \ \text{if}\ F\in \InnerEdges,\quad
	\restr{\avg{u}}{F} \coloneqq u\ \text{if}\ F\in \Edges\setminus\InnerEdges.
\end{equation*}
In a similar way, we define the jump operators
for matrix-valued functions $u\in H^1(\TT_h;\R^{d\times d})$ and for vector-valued functions $v\in H^1(\TT_h; \R^d)$ by
\begin{align*}
	\restr{\jump{u}}{F} &\coloneqq u^+\,n^+ + u^-\,n^-, &
	\restr{\jump{v}}{F} &\coloneqq v^+\otimes n^+ + v^-\otimes n^-,
	&&\text{if}\ F\in \InnerEdges,\\
	\restr{\jump{u}}{F} &\coloneqq u\,n, &
	\restr{\jump{v}}{F} &\coloneqq v\otimes n,
	&&\text{if}\ F\in \Edges\setminus\InnerEdges,
\end{align*}
with $n^+$ and $n^-$ the outward unit normal vectors on $\partial T^+$ and $\partial T^-$.
For scalar-valued functions we simply set $\restr{\jump{u}}{F} = u^+ - u^-$
for $F\in \InnerEdges$ and $\restr{\jump{u}}{F} = u$ for $F\in \Edges\setminus\InnerEdges$.

We will frequently use inverse inequalities and trace theorems in our analysis. These results are summarized in the following lemma.
\begin{lemma}\label{lem:invserse_trace}
	The following inequalities hold:
	\begin{enumerate}[label=\alph*)]
	\item For given $0\le \ell \le k \le 1$ and $s,t\in[1,\infty]$ there exists some $\Ctr>0$ depending on $k,\ell,s,t,p,d$ and $\kappa$ such that the inequality
	\begin{equation*}
		\abs{v_h}_{W^{k,s}(T)} \le \Ctr\, h_T^{k-\ell}\, \abs{T}^{1/s-1/t}\,\abs{v_h}_{W^{\ell,t}(T)},\quad \forall T\in\TT_h,
	\end{equation*}
	is fulfilled for all $v_h\in \PP_p(T)$.
	\item For given $s\in [1,\infty]$ there exists some $\Ctr>0$ depending on $s, p, d$ and $\kappa$ such that the inequality
	\begin{equation*}
		\norm{v_h}_{L^s(F)}\le \Ctr\,h_T^{-1/s}\,\norm{v_h}_{L^s(T)},\qquad \forall \Edges\ni F\subset T\in\TT_h,
	\end{equation*}
	is fulfilled for all $v_h\in \PP_p(T)$.
	\end{enumerate}
\end{lemma}
Note that we use the same notation for both constants in the previous lemma as they depend on the same quantities.

For our analysis we need the following broken Sobolev spaces
\begin{equation*}
	H_h^2(\Omega)\coloneqq \{v\in L^2(\Omega)\colon \restr{v}{T}\in H^2(T)\quad\forall T\in\TT_h\}.
\end{equation*}
Moreover, we introduce a mesh-dependent norm for the space $H_h^2(\Omega)\cap H_0^1(\Omega)$
\begin{equation}\label{eq:def_Hh2_norm}
	\Hnorm{v}^2 \coloneqq \sum_{T\in\TT_h}\norm{\Hessian v}_{L^2(T)}^2 + 
	\sum_{F \in \InnerEdges} h_F^{-1} \, \norm{\jump{\nabla v\cdot n_F}}_{L^2(F)}^2,
\end{equation}

\subsection{Approximation of the Hessian}
\label{sec:DiscretizationApproach}

Our discretization approach relies on a finite element approximation of the
Hessian of $u$ also referred to as \emph{Hessian recovery}.
For related ideas we refer to \cite{GuoZhangZhao2016} and the references therein.
In this article we study two different approaches.
The first approach uses an approximation with $C^0$-conforming finite
elements.
To illustrate the idea of the construction, consider the integration-by-parts formula for the second derivatives, \ie,
\begin{equation*}
	\int_\Omega \partial_{ij} u\,w\,\dx = -\int_\Omega \partial_i u\,\partial_j w\,\dx
	+ \int_\Gamma \partial_i u\,w\,n_j\,\ds\quad\forall w\in H^1(\Omega),
\end{equation*}
which is valid for all $u\in H^2(\Omega)$ and $i,j=1,\ldots,d$.
Here, $n(x)=(n_1(x),\ldots,n_d(x))^\top$ denotes the outer unit normal vector on $\Gamma$.
Alternatively, one can use the more compact equivalent formulation
\begin{equation*}
	\int_\Omega \nabla^2 u \dprod w\,\dx = - \int_\Omega \nabla u \cdot \Div w\,\dx
	+ \int_\Gamma \nabla u \cdot (w\,n)\,\ds \quad \forall w\in H^1(\Omega;\R^{d\times d}).
\end{equation*}
The Hessian approximation is sought in the finite-dimensional
space
\begin{equation*}
	\HSpaceCG{^{(p)}}(\R^{d\times d})
	\coloneqq
	\{w_h\in C(\overline\Omega;\R^{d\times d}) \colon \restr{w_h}{T}\in \PP_p^{d\times d}\quad \forall T\in \TT_h\}
\end{equation*}
with polynomial degree~$p\in \N$. To shorten the notation we will omit the superscript $(p)$, except when a different polynomial degree is used. The previous integral identity motivates the following definition. 
\begin{definition}[Continuous Galerkin Hessian]
	For each $u\in H_h^2(\Omega)$, the \emph{discrete Hessian} $\FEHessianCG{u}\in \HSpaceCG(\R^{d\times d})$ is defined by the variational problem
	\begin{equation}
		\label{eq:DiscreteHessianCG}
		\int_\Omega \FEHessianCG{u} \dprod w_h\,\dx
		=
		-\int_\Omega \nabla u \cdot \Div w_h\,\dx
		+ \int_\Gamma \nabla u \cdot (w_h\,n)\,\ds \quad \forall w_h\in \HSpaceCG(\R^{d\times d}).
	\end{equation}
\end{definition}
A further strategy is an approximation by piecewise polynomial but discontinuous functions.
To this end, we define the space
\begin{equation*}
	\HSpaceDG{^{(p)}}(\R^{d\times d})\coloneqq\{w_h\in L^\infty(\Omega; \R^{d\times d})\colon \restr{w_h}{T}\in \PP_p^{d\times d}\quad\forall T\in \TT_h\}.
\end{equation*}

We obtain a Hessian approximation by discretizing the element-wise integration-by-parts formula
\begin{align*}
	\int_\Omega\partial_{ij} u\,w\,\dx
	&= \sum_{T\in\TT_h}\left[-\int_T \partial_i u\,\partial_j w\,\dx
	+ \int_{\partial T} \partial_i u\,w\,n_{\partial T,j}\,\ds\right] \\
	&= -\sum_{T\in\TT_h}\int_T \partial_i u\,\partial_j w\,\dx
	+ \sum_{F\in \InnerEdges} \int_F \avg{\partial_i u}\,\jump{w}\,n_{F,j}\,\ds
	+ \sum_{F\in \Edges\setminus\InnerEdges} \int_F \partial_i u\,w\,n_j\,\ds
\end{align*}
which is valid for all $u\in H^2(\Omega)$ and $w\in H_h^1(\Omega)$.
This motivates the following definition:
\begin{definition}[Discontinuous Galerkin Hessian]
	For each $u\in H_h^2(\Omega)$, the \emph{DG Hessian} $\FEHessianDG{u}\in \HSpaceDG(\R^{d\times d})$ is defined by
	\begin{equation}
		\label{eq:DiscreteHessianDG}
		\int_\Omega \FEHessianDG{u} \dprod w_h\,\dx
		= - \sum_{T\in \TT_h} \int_T \nabla u\cdot \Div w_h\,\dx
		+ \sum_{F\in\Edges} \int_F \avg{\nabla u}\cdot\jump{w_h}\,\ds
	\end{equation}
	for all $w_h\in \HSpaceDG(\R^{d\times d})$.
\end{definition}

Many results in this article are independent of the choice of the Hessian
approximation. In this case we drop the superscript and simply write
$\FEHessian{}$ and $\HSpace$ which means either $\FEHessianCG{}$ and
$\HSpaceCG{}$ or $\FEHessianDG{}$ and $\HSpaceDG{}$.

We conclude this section with the following approximation result:
\begin{lemma}\label{lem:best_approximation_Hessian}
	Let $u\in H^2(\Omega)$ be given.
	The approximate Hessian $\FEHessian{u}$ (either $\FEHessianCG{u}$ or $\FEHessianDG{u}$)
	coincides with the $L^2(\Omega)$-projection of $\nabla^2 u$ onto $\HSpace(\R^{d\times d})$, \ie,
	\begin{equation*}
		\norm{\nabla^2 u - \FEHessian{u}}_{L^2(\Omega)} =
		\inf_{\W \in \HSpace(\R^{d\times d})} \norm{\nabla^2 u - \W}_{L^2(\Omega)}.
	\end{equation*}
	Moreover, there holds the stability estimate
	\begin{equation}\label{eq:stability_hessian}
		\norm{\FEHessian{u}}_{L^2(\Omega)} \le c_d\,\Ctr\,\Hnorm{u}\quad\forall u\in H_h^2(\Omega).
	\end{equation}
\end{lemma}
\begin{proof}
	The desired result follows from the definition \eqref{eq:DiscreteHessianCG}
	and the integration-by-parts formula which yields
	\begin{equation*}
		\int_\Omega \FEHessianCG{u} \dprod w_h\,\dx
		= -\int_\Omega \nabla u \cdot \Div w_h\,\dx
		+ \int_\Gamma \nabla u\cdot (w_h\,n)\,\ds
		= \int_\Omega \nabla^2 u \dprod w_h\,\dx
	\end{equation*}
	for all $w_h\in \HSpaceCG(\R^{d\times d})$. This implies that $\FEHessianCG{u}$ is the $L^2(\Omega)$-projection of $\nabla^2 u$ onto $\HSpaceCG(\R^{d\times d})$.
	To show the stability result we exploit the cell-wise integration-by-parts formula, taking into account the equality $\nabla u\cdot (w_h\,n) = (\nabla u\otimes n) \dprod w_h$, the Cauchy-Schwarz inequality and the discrete trace \cref{lem:invserse_trace}.
	We obtain for each $w_h\in \HSpaceCG(\R^{d\times d})$
	\begin{align*}
		\int_\Omega \FEHessianCG{u}\dprod w_h\,\dx
		&= - \int_\Omega \nabla u\cdot \Div(w_h)\,\dx + \int_\Gamma \nabla u\cdot (w_h\,n)\,\ds \\
		&= \sum_{T\in\TT_h} \int_T\nabla^2 u\dprod w_h\,\dx
		- \sum_{F\in\InnerEdges} \int_F \jump{\nabla u}\dprod w_h\,\ds \\
		&\le c_d\Biggh(){\sum_{T\in\TT_h} \norm{\nabla^2 u}_{L^2(T)}^2
		+ \Ctr^2\,\sum_{F\in\InnerEdges} h_F^{-1}\norm{\jump{\nabla u}}_{L^2(F)}^2}^{1/2}\,\norm{w_h}_{L^2(\Omega)}.
	\end{align*}
	With a simple computation taking into account that $u_h$ is continuous at the facets we deduce $\norm{\jump{\nabla u}}_{L^2(F)} = \norm{\jump{\nabla u\cdot n_F}}_{L^2(F)}$.
	Finally, we test the previous inequality with $w_h = \FEHessianCG{u}$ and
	divide the left- and right-hand side by $\norm{\FEHessianCG{u}}_{L^2(\Omega)}$ to conclude \eqref{eq:stability_hessian}.

	With similar arguments one can conclude the same results for the DG Hessian $\FEHessianDG{}$. The proof can be found in \cite[Lemma~2.1]{Neilan2017}.
\end{proof}

\subsection{A finite element scheme}
\label{sec:discretization_approach}

The finite element approximations of our problem \eqref{eq:strong_form_laplace} are sought in
the space of continuous Lagrange finite elements of order $p\in\N$, \ie,
\begin{equation*}
	\uSpace\coloneqq\{v_h\in C(\overline\Omega)\colon \restr{v_h}{T}\in \PP_p \quad \forall T\in\TT_h\},
\end{equation*}
and moreover, we define $\uSpaceZero = \uSpace\cap H_0^1(\Omega)$ to incorporate essential boundary conditions. 
The polynomial degree $p\ge 2$ is the same as for the space $\HSpace$. Later, we will see that this choice leads to an optimal balance of the approximation errors for the Hessian $\nabla^2u$ and the solution $u$.
Motivated by the strong formulation of the continuous problem \eqref{eq:strong_form_laplace}
we test the discrete equations with the finite element Laplacian
\begin{equation*}
	\FELaplace{v}\coloneqq \trace\,\FEHessian{v} = \sum_{i=1}^d\hii{v} \in \HSpace(\R).
\end{equation*}
The bilinear and linear forms we are going to use in the discrete scheme are defined by
\begin{align*}
	a_h(u_h, v_h)
	&\coloneqq
	\int_\Omega \gamma A \dprod \FEHessian{u_h}\,\FELaplace{v_h}\,\dx, \\
	F_h(v_h) &\coloneqq \int_\Omega \gamma f\,\FELaplace{v_h}\,\dx.
\end{align*}
The discrete problem reads
\begin{equation}
	\label{eq:DiscreteProblem}
	\text{Find}\ u_h \in \uSpaceZero \text{ such that } a_h(u_h,v_h) + J_h(u_h,v_h)
	=
	F_h(v_h)
	\quad\forall v_h \in \uSpaceZero
	.
\end{equation}
The bilinear form $J_h\colon \uSpaceZero\times \uSpaceZero\to\R$
may contain several stabilization terms in order to guarantee discrete coercivity. The specific form of the stabilization terms will be introduced later.

The nodal basis functions of $\uSpaceZero$ and $\HSpace$ are denoted by
\begin{equation*}
	\uSpaceZero = \Span\{\varphi_\ell\}_{\ell=1}^{N_V},\qquad
	\HSpace = \Span\{\psi_\ell\}_{\ell=1}^{N_W}.
\end{equation*}
For some function $u_h\in \uSpaceZero$ we denote by
$\bu = (u_1,\ldots,u_{N_V})^\top$ the coefficient vector
satisfying $u_h = \sum_{\ell=1}^{N_V} u_\ell\,\varphi_\ell$.
Analogously, we use the vector representation $\bh_{ij}=(h_{ij,1},\ldots,h_{ij,N_W})^\top$ for
the Hessian approximations $\FEHessian{u_h}$, \ie,
$\hij{u_h} = \sum_{\ell=1}^{N_W} h_{ij,\ell}\,\psi_\ell$ for $i,j=1,\ldots,d$.

To realize our algorithm with $\FEHessian{} = \FEHessianCG{}$ we first assemble the matrices and load vector
\begin{align*}
	M_W &= \varh(){\int_\Omega \psi_\ell\,\psi_k}_{k,\ell} \in\R^{N_W\times N_W}, && \text{(mass matrix in $\HSpace$)}\\
	C_{ij} &= \varh(){- \int_\Omega \frac{\partial \varphi_\ell}{\partial x_i}\, \frac{\partial \psi_k}{\partial x_j}
	+ \int_\Gamma \frac{\partial \varphi_\ell}{\partial x_i} \,\psi_k\, n_j}_{k,\ell}\in \R^{N_W\times N_V}, && \text{(partial mixed stiffness matrix)} \\
	B_{ij} &= \varh(){\int_\Omega \gamma A_{ij} \,\psi_\ell\,\psi_k}_{k,\ell}\in\R^{N_W\times N_W},
	&& \text{(weighted mass matrix in $\HSpace$)} \\
	S &= J_h\varh(){\varphi_\ell,\varphi_k}_{k,\ell}\in\R^{N_V\times N_V},
	&& \text{(stabilization matrix)} \\
	\bf_{W} &= \varh(){\int_\Omega \gamma f\,\psi_k}_{k}\in\R^{N_W}. && \text{(load vector w.r.t.\ $\HSpace$)}
\end{align*}
In the case $\FEHessian{} = \FEHessianDG{}$ the matrices $C_{ij}$ have to be modified according to the right-hand side of \eqref{eq:DiscreteHessianDG}. Moreover, the dimension of the  matrices increases as the number of degrees of freedom $N_W$ is higher for the function space $\HSpaceDG$.
Obviously, the equations \eqref{eq:DiscreteHessianCG} or \eqref{eq:DiscreteHessianDG}
with $u$ replaced by $u_h$ can be expressed by means of
\begin{equation}
	\label{eq:EqSystem_Hessian}
	M_W\,\bh_{ij} = C_{ij}\,\bu.
\end{equation}
The application of $\FELaplace{}$ to the test function $v_h\in \uSpaceZero$
represented by the coefficient vector $\bv\in \R^{N_V\times N_V}$ leads to a new function
$w_h\coloneqq\FELaplace{v_h}$ with coefficient vector
$\bw\in\R^{N_W}$ and can be computed by means of
\begin{equation}
	\label{eq:EqSystem_TestFunction}
	M_W\,\bw = \sum_{i=1}^d C_{ii}\, \bv.
\end{equation}
The right-hand side of \eqref{eq:DiscreteProblem} can be evaluated by means of
\begin{equation}\label{eq:Rhs_discrete}
    F_h(v_h) = \int_\Omega \gamma f\, \FELaplace{v_h}\,\dx = \bv^\top \Bigh(){\sum_{i=1}^d C_{ii}}^\top M_W^{-1} \bf_W.
\end{equation}
A representation for the left-hand side follows after insertion of \eqref{eq:EqSystem_Hessian}--\eqref{eq:Rhs_discrete} into \eqref{eq:DiscreteProblem}.
This yields
\begin{equation}
	\label{eq:System_u}
	a_h(u_h, v_h) + J_h(u_h,v_h) 
	= 
	\bv^\top \, \Bigh(){\sum_{i=1}^d C_{ii}}^\top \, M_W^{-1} \, \Bigh(){\sum_{i,j=1}^d B_{ij}\, M_W^{-1}\, C_{ij}} \bu 
	+ 
	\bv^\top S\,\bu.
\end{equation}
Consequently, problem \eqref{eq:DiscreteProblem} is equivalent to
\begin{equation}
	\label{eq:discrete_system}
	\varh(){ \Bigh(){\sum_{i=1}^d C_{ii}}^\top \, M_W^{-1} \, \Bigh(){\sum_{i,j=1}^d B_{ij}\, M_W^{-1}\, C_{ij}} + S} \bu
	=
	\Bigh(){\sum_{i=1}^d C_{ii}}^\top M_W^{-1} \bf_W
	\eqqcolon \bf_V
	.
\end{equation}
Although the system matrix cannot be assembled explicitly, one can compute matrix-vector products, each of 
which requires the solution of $d^2+1$ linear equation systems for the mass matrix $M_W$.
In our numerical experiments we precomputed an LU factorization of $M_W$.
Each evaluation of $M_W^{-1}$ then corresponds to an inexpensive forward-backward substitution.
The non-symmetric system \eqref{eq:discrete_system} can be efficiently solved by a preconditioned \gmres\ algorithm.
As a preconditioner we utilize the matrix
\begin{equation}
	\label{eq:preconditioner}
    P = \Bigh(){\sum_{i=1}^d C_{ii}}^\top \, \widehat M_W^{-1}\sum_{i,j=1}^d  \widehat B_{ij}\, \widehat M_W^{-1}\, C_{ij}+S,
\end{equation}
where $\widehat M_W^{-1}$ is the inverse of the main diagonal of $M_W$ and $\widehat B_{ij}$ is the main diagonal of $B_{ij}$.
This allows us to assemble~$P$ explicitly.
Employing $\widehat B_{ij}$ instead of $B_{ij}$ yields a sparser preconditioner and in case of a problem with vanishing off-diagonal entries of $A$, i.e., $A_{ij} = 0$ for $i \neq j$, a symmetric preconditioner $P$.
A direct solver is then used to solve the systems of linear equations involving $P$.
Note that it is not appropriate to use the lumped mass matrix as this might yield a singular
matrix whenever the polynomial degree of the space $\HSpace$ is larger than one.
The numerical experiments conducted in \cref{sec:experiments} indicate that the preconditioned \gmres\ method for~\eqref{eq:discrete_system} is robust with respect to mesh refinement. For a more sophisticated preconditioning strategy for non-divergence form PDEs we refer to \cite{Smears2018}, where a domain decomposition preconditioner is studied.

An alternative viable strategy is the solution of a block system equivalent to \eqref{eq:discrete_system}.
This becomes particularly useful if, in addition to the solution vector $\bu$, one is interested in the finite-element Hessian, \eg, for the solution of HJB equations.
To this end, we use the substitution
from \eqref{eq:EqSystem_Hessian} as well as
\begin{equation*}
	M_W\,\bp = \sum_{i,j=1}^d B_{ij}\,\bh_{ij}
\end{equation*}
and arrive (in case $d=2$) at the equation system
\begin{equation*}
	\begin{pmatrix}
		M_W & & & & & -C_{11} \\
		& M_W & & & & -C_{12} \\
		& & M_W & & & -C_{21} \\
		& & & M_W & & -C_{22} \\
		-B_{11} & -B_{12} & -B_{21} & -B_{22} & M_W & \\
		& & & & C_{11}^\top + C_{22}^\top & S
	\end{pmatrix}
	\begin{pmatrix}
		\bh_{11} \\ \bh_{12} \\
		\bh_{21} \\ \bh_{22} \\
		\bp \\ \bu
	\end{pmatrix}
	=
	\begin{pmatrix}
		0 \\ 0 \\ 0 \\ 0 \\ 0 \\ \bf_V
	\end{pmatrix}
\end{equation*}
equivalent to~\eqref{eq:discrete_system}.
The modification for the three-dimensional case is obvious.

\subsection{Well-posedness of the discrete scheme}

The scheme \eqref{eq:DiscreteProblem} can be interpreted as a non-conforming discretization of the
variational problem \eqref{eq:weak_form} as the usage of approximate Hessians and Laplacians implies $a\ne a_h$ and $F\ne F_h$, and there also holds $\uSpaceZero\not\subset X=H^2(\Omega)\cap H_0^1(\Omega)$.
\begin{lemma}
	\label{lem:H2h_norm_bound}
	The inequality
	\begin{equation}\label{eq:discrete_norm_bound}
		\Hnorm{v}^2 \le c_{\kappa,d} \, 
		\sum_{T\in\TT_h} \varh(){h_T^{-2}\,\norm{\nabla v}_{L^2(T)}^2
		+ \norm{\Hessian v}_{L^2(T)}^2}
	\end{equation}
	is valid for all $v\in H_h^2(\Omega)$.
	Furthermore, $\Hnorm{\cdot}$ is a norm in $X=H^2(\Omega) \cap H_0^1(\Omega)$.
\end{lemma}
\begin{proof}
	To show \eqref{eq:discrete_norm_bound} we merely have to discuss the jump
	terms in the definition \eqref{eq:def_Hh2_norm}.
	To this end, we apply the triangle inequality
	\begin{equation}\label{eq:splitting_jump_term}
		\sum_{F \in \InnerEdges} h_F^{-1} \norm{\jump{\nabla v\cdot n_F}}_{L^2(F)}^2
		\le \sum_{F \in \InnerEdges}
		\sum_{\genfrac{}{}{0pt}{}{T\in\TT_h}{F\subset T}}
		h_F^{-1} \, \norm{\restr{\nabla v}{T}}_{L^2(F)}^2
	\end{equation}
	and a trace theorem on a reference setting
	\begin{align*}
		\norm{\restr{\nabla v}{T}}_{L^2(F)}
		&\le c_\kappa \left(h_F^{-1/2}\,\norm{\nabla v}_{L^2(T)} + h_F^{1/2}\,\norm{\nabla^2 v}_{L^2(T)}\right).
	\end{align*}
	Using also the assumed shape regularity, which implies $c_\kappa^{-1}\,h_T\le h_F\le c_\kappa\, h_T$ for $F\subset T$, we infer \eqref{eq:discrete_norm_bound}. The fact that $\Hnorm{\cdot}$ is a norm in $H^2(\Omega)\cap H_0^1(\Omega)$ follows from standard arguments.
\end{proof}

The main ingredient for the proof of the existence result for strong solutions (\cref{lem:existence}) is a Miranda-Talenti estimate of the form $\abs{u}_{H^2(\Omega)} \le \norm{\Delta u}_{L^2(\Omega)}$ which is valid, \eg, if the underlying domain $\Omega$ is convex.
To show well-posedness of our discrete scheme we first have to prove a discrete Miranda-Talenti estimate. 
A similar result, but for a discretization using the element-wise exact Hessian and Laplacian, is proved in \cite[Theorem~1]{NeilanWu2019}. 
We begin with the following auxiliary result.
\begin{lemma}\label{lem:lifting_error_CG}
	For each polynomial degree $p\in \N$, there exists a lifting operator $\LiftPP\colon\HSpaceDG \to \HSpaceCG$ satisfying the estimate
	\begin{equation*}
		\norm{v_h - \LiftPP(v_h)}_{L^2(\Omega)}
		\le C\Biggh(){\sum_{F\in \InnerEdges} h_F\,\norm{\jump{v_h}}_{L^2(F)}^2}^{1/2}
	\end{equation*}
	for all $v_h\in \HSpaceDG$, where $C>0$ depends on $d, p, \Cv, \Ctr$ and $\kappa$.
\end{lemma}
\begin{proof}
	The proof is similar to the proof of \cite[Theorem 2.2]{KarakashianPascal2003}
	and \cite[Theorem~2.2]{Burman2005}, but in both articles a slightly
	different setting is considered. For the convenience of the reader
	we repeat the proof.

	We denote by $x_{T,i}$ the Lagrange points of the local finite element $(T,\PP_{p},\Sigma)$.
	That is, the functionals $\sigma_{T,i}\in\Sigma$, $i=1,\ldots,s$, with $s\coloneqq \frac12\,(p+1)\,(p+2)$ if $d=2$ and $s\coloneqq\frac16\,(p+1)\,(p+2)\,(p+3)$ if $d=3$, have the form $\sigma_{T,i}(v) = v(x_{T,i})$.
	The sets $\TT_{T,i}\coloneqq\{T\in\TT_h\colon x_{T,i}\in T\}$ contain all elements of $\TT_h$ sharing the Lagrange point $x_{T,i}$ and we denote the number of cells in $\TT_{T,i}$ by $\#\TT_{T,i}$.
	In a similar way we define the patch of facets $\FF_{T,i}\coloneqq\{F\in\InnerEdges\colon x_{T,i}\in F\}$ having $x_{T,i}$ as a vertex.
	Moreover, $\{\varphi_{T,i}\}_{i=1}^s$ is the nodal basis of $(T,\PP_{p},\Sigma)$, \ie, $\sigma_{T,i}(\varphi_{T,j}) = \delta_{ij}$ holds for all $i,j=1,\ldots,s$. 
	The precise definition of our lifting operator is
	\begin{equation*}
		\restr{\LiftPP(v_h)}{T}
		\coloneqq \sum_{i=1}^s \Biggh(){\frac{1}{\# \TT_{T,i}} \sum_{T'\in\TT_{T,i}} (\restr{v_h}{T'})(x_{T,i})}\,\varphi_{T,i}.
	\end{equation*}
	Next, we derive local estimates for the lifting error on a single element $T\in \TT_h$.
	From the definition of $\LiftPP$ and the triangle inequality we conclude
	\begin{align}
		\label{eq:lifting_error_est_1}
		\norm{v_h - \LiftPP(v_h)}_{L^2(T)}^2
		&
		= 
		\int_T \left(\sum_{i=1}^s \Biggh(){(\restr{v_h}{T})(x_{T,i}) - \frac{1}{\# \TT_{T,i}} \sum_{T'\in\TT_{T,i}} (\restr{v_h}{T'})(x_{T,i})} \varphi_{T,i}\right)^2 \dx 
		\nonumber\\
		& 
		\le \int_T \Biggh(){\sum_{i=1}^s \frac1{\#\TT_{T,i}}\sum_{T'\in \TT_{T,i}}
		\Bigabs{\restr{v_h}{T}-\restr{v_h}{T'}}(x_{T,i})\, \abs{\varphi_{T,i}}}^2 \dx.
	\end{align}
	We distinguish several cases:
	if $x_{T,i}$ is a Lagrange point in the interior of $T$ or in the interior of a boundary facet $F\in \Edges$ with $F\subset\Gamma$, then $\TT_{T,i} = \{T\}$ holds and consequently
	\begin{equation*}
		\sum_{T'\in \TT_{T,i}}\Bigabs{\restr{v_h}{T}-\restr{v_h}{T'}}(x_{T,i}) = 0.
	\end{equation*}
	If $x_{T,i}$ is located in the interior of an inner facet~$F\in\InnerEdges$, there holds $\TT_{T,i} = \{T,T'\}$ with $F=T\cap T'$ and we obtain together with the inverse inequality from \cref{lem:invserse_trace}
	\begin{equation*}
		\Bigabs{\restr{v_h}{T}-\restr{v_h}{T'}}(x_{T,i})
		\le \norm{\jump{v_h}}_{L^\infty(F)} \le \Ctr\,h_F^{-(d-1)/2}\,\norm{\jump{v_h}}_{L^2(F)}.
	\end{equation*}
	If $x_{T,i}$ coincides with a vertex of $T$ or, in the case $d=3$, is located at an edge of $T$, we choose a sequence of simplices $T=T_1, T_2,\ldots,T_\ell=T'$ such that $T_j$ and $T_{j+1}$, $j=1,\ldots,\ell-1$, share a common facet $F_j\in \FF_{T,i}$.
	With the triangle inequality and similar arguments like in the previous case we deduce
	\begin{equation*}
		\abs{\restr{v_h}{T} - \restr{v_h}{T'}}(x_{T,i})\le \sum_{j=1}^{\ell-1} \abs{\restr{v_h}{T_j} - \restr{v_h}{T_{j+1}}}(x_{T,i})
		\le \Ctr\,\sum_{F\in \FF_{T,i}} h_F^{-(d-1)/2}\norm{\jump{v_h}}_{L^2(F)}.
	\end{equation*}
	We summarize the previous cases and infer
	\begin{align*}
		\frac1{\#\TT_{T,i}}\sum_{T'\in \TT_{T,i}}\Bigabs{\restr{v_h}{T}-\restr{v_h}{T'}}(x_{T,i})
		&\le
		\Ctr\,\sum_{F\in \FF_{T,i}} h_F^{-(d-1)/2}\norm{\jump{v_h}}_{L^2(F)}.
	\end{align*}
	Insertion into \eqref{eq:lifting_error_est_1} yields together with
	the discrete Cauchy-Schwarz inequality
	\begin{align*}
		\norm{v_h - \LiftPP(v_h)}_{L^2(T)}^2
		&\le C_\FF\,s\,\Ctr^2\,\sum_{i=1}^s\sum_{F\in \FF_{T,i}} h_F^{-(d-1)}\,\norm{\jump{v_h}}_{L^2(F)}^2 \int_T \varphi_{T,i}^2\,\dx \\
		&\le c_{d,\kappa}\,s\,\Ctr^2\,\Cv^d\,\sum_{i=1}^s\sum_{F\in \FF_{T,i}}h_F\,\norm{\jump{v_h}}_{L^2(F)}^2,
	\end{align*}
	 with $C_\FF:=\max_{T\in\TT_h}\max_{i=1,\ldots,s} \# \FF_{T,i}\le c_{d,\kappa}$. The last step follows from $\int_T \varphi_{T,i}^2 \, \dx \le \abs{T}\le c_\kappa\,h_T^d\le c_\kappa\,\Cv^d\,h_F^d$.
	Summation over all $T\in \TT_h$ leads to the assertion.
\end{proof} {
}
\begin{lemma}[Discrete Miranda-Talenti estimate]\label{lem:Discrete_Miranda_Talenti_with_Jumps}
	Let $\Omega\subset\R^d$ be a bounded and convex
	domain. The polynomial degree of $\uSpace$ and
	$\HSpace$ is $p\ge 2$. There exist constants $C_1,C_2>0$
	depending on $\Ctr,\Cv,\kappa$ and $d$ and, if $d=3$, also on $\Omega$, such that for each $u_h\in \uSpace$ the inequalities
	\begin{align}
		\label{eq:Discrete_Miranda_Talenti}
		\norm{\FEHessianCG{u_h}}_{L^2(\Omega)}^2 &\le \norm{\FELaplaceCG{u_h}}_{L^2(\Omega)}^2
		+ C_1\,\sum_{F\in\InnerEdges} h_F^{-1}\,\norm{\jump{\nabla u_h\cdot n_F}}_{L^2(F)}^2 \nonumber\\
		&\phantom{\le \norm{\FELaplaceCG{u_h}}_{L^2(\Omega)}^2\ } + C_2\,\sum_{F\in\InnerEdges} h_F\,\norm{\jump{\nabla^2 u_h}}_{L^2(F)}^2, \\
		\label{eq:Discrete_Miranda_TalentiDG}
		\norm{\FEHessianDG{u_h}}_{L^2(\Omega)}^2 &\le \norm{\FELaplaceDG{u_h}}_{L^2(\Omega)}^2
		+ C_1\,\sum_{F\in\InnerEdges} h_F^{-1}\,\norm{\jump{\nabla u_h\cdot n_F}}_{L^2(F)}^2
	\end{align}
	are fulfilled.
\end{lemma}
\begin{proof}	
	We first introduce a further lifting operator $\LiftHCT\colon \uSpace\to \uSpaceConf$ which maps $u_h$ into an $H^2$-conforming finite element space $\uSpaceConf$. In the case $d=2$, we will make use of the space generated by the Hsieh-Clough-Tocher (HCT) element \cite{Ciarlet1978} or some higher-order analogue.
	The lifting operator $\LiftHCT$ fulfills the estimate
	\begin{equation}\label{eq:H2_lifting_estimate}
		\norm{\nabla_h^k (u_h-\LiftHCT(u_h))}_{L^2(\Omega)}
		\le \CE\,
		\Biggh(){\sum_{F\in\InnerEdges} h_F^{3-2k}\norm{\jump{\nabla u_h\cdot n_F}}_{L^2(F)}^2}^{1/2},\quad k=0,1,2,
	\end{equation}
	with a constant $\CE>0$ depending on $\Cv,\Ctr,\kappa$ and $d$, but for the case $d=3$ also on the structure of $\Omega$. In particular, if an opening angle at a sharp edge $\alpha$ of $\Omega$ tends to $\pi$, then $\CE\to \infty$. 

	A proof of \eqref{eq:H2_lifting_estimate} in the two-dimensional case can be found in \cite[Section~4.11.3]{Verfuerth2013}, \cite[Equation~(2.9)]{BrennerGudiSung2010} for the case $p=2$ and in \cite[Lemma~3.1]{GeorgoulisHoustonVirtanen2011} for $p\ge 2$.
	For the three-dimensional case we refer to \cite{NeilanWu2019}, where a 3D HCT element for polynomial degrees $p\in\{2,3\}$ is studied and to \cite{BrennerSung2019_preprint}, where a different function space based on virtual elements of arbitrary order is used.

	We set $\uconf\coloneqq \LiftHCT(u_h)$ and obtain with the triangle inequality
	\begin{equation}\label{eq:MD_est_splitting_1}
		\norm{\FEHessian{u_h}}_{L^2(\Omega)}
		\le \norm{\nabla^2 \uconf}_{L^2(\Omega)}
		+ \norm{\nabla_h^2 \uconf - \nabla_h^2 u_h}_{L^2(\Omega)}
		+ \norm{\nabla_h^2 u_h - \FEHessian{u_h}}_{L^2(\Omega)}.
	\end{equation}
	For the first term on the right-hand side we can directly apply the continuous Miranda-Talenti estimate from \cite[Theorem~2]{SmearsSueli2013}.
	After insertion of further intermediate functions we obtain
	\begin{align}\label{eq:MD_est_splitting_2}
		\norm{\nabla^2 \uconf}_{L^2(\Omega)}
		&\le \norm{\Delta \uconf}_{L^2(\Omega)}\nonumber\\
		&\le \norm{\FELaplace{u_h}}_{L^2(\Omega)}
		+ \norm{\FELaplace{u_h}-\Delta_h u_h}_{L^2(\Omega)}
		+ \norm{\Delta_h u_h - \Delta_h \uconf}_{L^2(\Omega)}.
	\end{align}
	It remains to bound the two last terms on the right-hand sides of \eqref{eq:MD_est_splitting_1} and \eqref{eq:MD_est_splitting_2}.
	From the error estimate \eqref{eq:H2_lifting_estimate} we infer
	\begin{equation}\label{eq:lifting_error_estimate}
		\norm{\Delta_h u_h - \Delta_h \uconf}_{L^2(\Omega)}
		+ \norm{\nabla_h^2 \uconf - \nabla_h^2 u_h}_{L^2(\Omega)}
		\le
		\CE\,\Biggh(){\sum_{F\in\InnerEdges}h_F^{-1} \norm{\jump{\nabla u_h\cdot n_F}}_{L^2(F)}^2}^{1/2}.
	\end{equation}
	In order to prove a bound for the approximations $\FELaplace{u_h}$ of $\Delta_h u_h$ and
	$\FEHessian{u_h}$ of $\nabla^2_h u_h$ we introduce the $L^2(\Omega)$-projection $\P_{\HSpace}$
	onto $\HSpace(\R^{d\times d})$ and obtain 
	\begin{align}
		\label{eq:error_exact_FE_hessian_1}
		\norm{\FEHessian{u_h} - \nabla_h^2 u_h}_{L^2(\Omega)}
		&\le \norm{\FEHessian{u_h} - \P_{\HSpace}(\nabla_h^2 u_h)}_{L^2(\Omega)}
		+ \norm{\P_{\HSpace}(\nabla_h^2 u_h) - \nabla_h^2 u_h}_{L^2(\Omega)}.
	\end{align}

	To bound the first term on the right-hand side of \eqref{eq:error_exact_FE_hessian_1}
	in case of $\FEHessian{} = \FEHessianCG{}$ we test \eqref{eq:DiscreteHessianCG} with the function $w_h\coloneqq\FEHessianCG{u_h} - \P_{\HSpaceCG}(\nabla_h^2 u_h)\in \HSpaceCG(\R^{d\times d})$, apply the orthogonality of $\P_{\HSpaceCG}$, the definition of $\FEHessianCG{}$, the integration-by-parts formula and \cref{lem:invserse_trace} to arrive at
	\begin{align}
		\label{eq:error_exact_FE_hessian_2_CG}
		&\norm{\FEHessianCG{u_h} - \P_{\HSpaceCG}(\nabla_h^2 u_h)}_{L^2(\Omega)}^2
		= \int_\Omega \left(\FEHessianCG{u_h} - \nabla_h^2 u_h\right)\dprod w_h\,\dx\nonumber\\
		&\quad= 
		\int_\Gamma \nabla u_h\cdot (w_h\,n)\,\ds
		- \sum_{T\in\TT_h} 
		\int_{\partial T} \nabla u_h\cdot (w_h\,n_{\partial T})\,\ds \nonumber\\
		&\quad= -\sum_{F\in \InnerEdges} \int_F \jump{\nabla u_h}\dprod w_h\,\ds
		\le \Ctr\, \varh(){\sum_{F\in \InnerEdges} h_F^{-1}\norm{\jump{\nabla u_h\cdot n_F}}_{L^2(F)}^2}^{1/2} \norm{w_h}_{L^2(\Omega)}.
	\end{align}
	Note that we used the relation $w_h\colon (n_F\otimes \nabla u_h) = \nabla u_h\cdot (w_h\,n_F)$ as well as the fact that the jumps in tangential direction $\jump{\nabla u_h\cdot t_F}$ vanish as $u_h$ is continuous along the facets $F$.

	In the case $\FEHessian{} = \FEHessianDG{}$ we use similar arguments, in particular the integration-by-parts formula and \eqref{eq:DiscreteHessianDG}, to obtain
	\begin{align}
		\label{eq:error_exact_FE_hessian_2_DG}
		&\norm{\FEHessianDG{u_h} - \P_{\HSpaceDG}(\nabla_h^2 u_h)}_{L^2(\Omega)}^2
		= \int_\Omega \left(\FEHessianDG{u_h} - \nabla_h^2 u_h\right)\dprod w_h\,\dx\nonumber\\
		&\quad= 
		\sum_{F\in\Edges} \int_F \avg{\nabla u_h}\cdot \jump{w_h}\,\ds
		-\sum_{T\in\TT_h} \int_{\partial T} \nabla u_h\cdot (w_h\,n_{\partial T})\,\ds \nonumber\\
		&\quad= -\sum_{F\in \InnerEdges} \int_F\jump{\nabla u_h} \dprod\avg{w_h}\,\ds
		\le \Ctr \, \varh(){\sum_{F\in \InnerEdges} h_F^{-1}\norm{\jump{\nabla u_h\cdot n_F}}_{L^2(F)}^2}^{1/2} \norm{w_h}_{L^2(\Omega)}.
	\end{align}
	Next, we discuss the second term on the right-hand side of \eqref{eq:error_exact_FE_hessian_1}.
	In case of $\FEHessian{} = \FEHessianCG{}$ we obtain an estimate from \cref{lem:lifting_error_CG} and the property $\nabla_h^2 u_h\in \MyDGSpace(\R^{d\times d})\subset \HSpaceDG(\R^{d\times d})$, \ie,
	\begin{align}
		\label{eq:CG_lifting_error_estimate}
		\norm{\P_{\HSpaceCG}(\nabla_h^2 u_h) - \nabla_h^2 u_h}_{L^2(\Omega)}
		&\le \norm{\LiftPP(\nabla_h^2 u_h) - \nabla_h^2 u_h}_{L^2(\Omega)}\nonumber\\
		&\le C\,\Biggh(){\sum_{F\in \InnerEdges} h_F\,\sum_{i,j=1}^d \norm{\jump{\partial_{ij} u_h}}_{L^2(F)}^2}^{1/2}\nonumber\\
		&\le C\,\Biggh(){\sum_{F\in \InnerEdges} h_F\, \norm{\jump{\nabla^2 u_h}}_{L^2(F)}^2}^{1/2}.
	\end{align}
	Note that the jump operator for matrix-valued functions involves only jumps in normal direction. In order to confirm the last step in the previous estimate, one just has to take into account that $u_h$ is continuous along the element facets so that the tangential components of the jumps vanish.
	
	Finally, one observes that the second term on the right-hand side of
	\eqref{eq:error_exact_FE_hessian_1} vanishes in case of $\FEHessian{} = \FEHessianDG{}$, \ie,
	\begin{equation}\label{eq:error_est_l2_proj_DG}
		\norm{\P_{\HSpaceDG}(\nabla_h^2 u_h) - \nabla_h^2 u_h}_{L^2(\Omega)} = 0,
	\end{equation}
	which is due to the fact that $\nabla_h^2 u_h\in \HSpaceDG(\R^{d\times d})$ holds.

	The discrete Miranda-Talenti estimates follow after inserting 
	\eqref{eq:error_exact_FE_hessian_2_CG} and \eqref{eq:CG_lifting_error_estimate}
	in case of $\FEHessian{} = \FEHessianCG{}$, or
	\eqref{eq:error_exact_FE_hessian_2_DG} and \eqref{eq:error_est_l2_proj_DG}
	in case of $\FEHessian{} = \FEHessianDG{}$, into \eqref{eq:error_exact_FE_hessian_1}, and combining the resulting estimates with
	\eqref{eq:MD_est_splitting_1} and \eqref{eq:MD_est_splitting_2}.
\end{proof}
The following result is needed in order to treat the fact that we work with
different norms for the spaces $\uSpace$ which have to be bounded by each other.
\begin{lemma}\label{eq:Bound_Laplace_FELaplace}
The following estimates are valid for arbitrary $u_h\in \uSpace$:
\begin{align*}
    \norm{\FELaplaceCG{u_h}}_{L^2(\Omega)}^2 &\ge \frac13\,\norm{\Delta_h u_h}_{L^2(\Omega)}^2 - \frac{d+1}3\,\Ctr^2\,\sum_{F\in\InnerEdges}h_F^{-1}\norm{\jump{\nabla u_h\cdot n_F}}_{L^2(F)}^2\\
	&\qquad - \frac{d+1}6\,\CE^2\,\sum_{F\in\InnerEdges}h_F\,\norm{\jump{\Delta u_h}}_{L^2(F)}^2,\\[.5em]
\norm{\FELaplaceDG{u_h}}_{L^2(\Omega)}^2 &\ge \frac13\,\norm{\Delta_h u_h}_{L^2(\Omega)}^2 - \frac{d+1}{3}\,\Ctr^2\,\sum_{F\in\InnerEdges}h_F^{-1}\,\norm{\jump{\nabla u_h\cdot n_F}}_{L^2(F)}^2.
\end{align*}
\end{lemma}
\begin{proof}
We start with the case $\FEHessian{} = \FEHessianDG{}$.
We apply the definition \eqref{eq:DiscreteHessianDG} and the integration-by-parts formula and obtain for arbitrary $w_h\in \HSpaceDG$
\begin{align*}
	(\FELaplaceDG{u_h}, w_h)_{L^2(\Omega)}
	&= -(\nabla_h u_h,\nabla_h w_h)_{L^2(\Omega)}
	+ \sum_{F\in \Edges} \int_F \avg{\nabla u_h\cdot n_F}\, \jump{w_h}\,\ds \\
	&= (\Delta_h u_h,w_h)_{L^2(\Omega)}
	- \sum_{F\in\InnerEdges} \int_F \jump{\nabla u_h\cdot n_F}\,\avg{w_h}\,\ds.
\end{align*}
Testing this equation with $w_h = \FELaplaceDG{u_h}\in \HSpaceDG$ yields together with a further application of the integration-by-parts formula
\begin{align*}
	&\norm{\FELaplaceDG{u_h}}_{L^2(\Omega)}^2
	= (\Delta_h u_h, \FELaplaceDG{u_h})_{L^2(\Omega)}
	- \sum_{F\in\InnerEdges} \int_F \jump{\nabla u_h\cdot n_F}\,\avg{\FELaplaceDG{u_h}}\,\ds \\
	&\quad= -(\nabla_h\Delta_h u_h, \nabla_h u_h)_{L^2(\Omega)} - \sum_{F\in\InnerEdges} \int_F \varh(){\jump{\nabla u_h\cdot n_F}\,\avg{\FELaplaceDG{u_h}} - \avg{\nabla u_h\cdot n_F}\,\jump{\Delta u_h}}\,\ds \\
	&\qquad + \sum_{F\in\Edges\setminus\InnerEdges}\int_F\nabla u_h\cdot n_F\,\Delta u_h\,\ds\\
	&\quad= (\Delta_h u_h,\Delta_h u_h)_{L^2(\Omega)} - \sum_{F\in\InnerEdges} \int_F \varh(){\jump{\nabla u_h\cdot n_F}\,\avg{\FELaplaceDG{u_h}} - \avg{\nabla u_h\cdot n_F}\,\jump{\Delta u_h}}\,\ds \\
	&\qquad + \sum_{F\in\Edges\setminus\InnerEdges}\int_F\nabla u_h\cdot n_F\,\Delta u_h\,\ds - \sum_{F\in \Edges} \int_F \jump{\nabla_h u_h\cdot n_F\,\Delta u_h}\,\ds.
\end{align*}
Together with the identity
\begin{equation*}
	\jump{\nabla_h u_h\cdot n_F\,\Delta u_h} = \avg{\nabla u_h\cdot n_F}\,\jump{\Delta u_h} + \jump{\nabla u_h\cdot n_F}\,\avg{\Delta u_h}
\end{equation*}
for all $F\in \InnerEdges$ we arrive at
\begin{equation*}
	\norm{\FELaplaceDG{u_h}}_{L^2(\Omega)}^2 = \norm{\Delta_h u_h}_{L^2(\Omega)}^2
	- \sum_{F\in\InnerEdges} \int_F \jump{\nabla u_h\cdot n_F}\,\avg{\FELaplaceDG{u_h} + \Delta_h u_h}\,\ds.
\end{equation*}
With the Cauchy-Schwarz and the Young inequality using also the discrete trace inequality from \cref{lem:invserse_trace}, we then deduce
\begin{align*}
	\norm{\FELaplaceDG{u_h}}_{L^2(\Omega)}^2
	&\ge \norm{\Delta_h u_h}_{L^2(\Omega)}^2
	- \xi \sum_{F\in\InnerEdges} h_F^{-1}\,\norm{\jump{\nabla u_h\cdot n_F}}_{L^2(F)}^2 \\
	&\qquad- \frac{\Ctr^2\,(d+1)}{4\,\xi} \varh(){\norm{\FELaplaceDG{u_h}}_{L^2(\Omega)}^2
	+ \norm{\Delta_h u_h}_{L^2(\Omega)}^2}
\end{align*}
for arbitrary $\xi>0$. We use the choice $\xi=\Ctr^2\,(d+1)/2$ and after rearrangement of the above inequality we arrive at
\begin{equation}
\label{eq:Proof_Bound_HLaplace_CG}
	\norm{\FELaplaceDG{u_h}}_{L^2(\Omega)}^2 \ge \frac13 \norm{\Delta_h u_h}^2 - \frac13(d+1)\Ctr^2 \sum_{F\in\InnerEdges}h_F^{-1}\norm{\jump{\nabla u_h\cdot n_F}}_{L^2(F)}^2.	
\end{equation}

In a similar way we derive the estimate for $\FELaplace{}=\FELaplaceCG{}$. First, we use the definition \eqref{eq:def_discrete_Hessian} and the integration-by-parts formula and obtain for each $w_h\in \HSpaceCG$
\begin{align*}
	(\FELaplaceCG{u_h}, w_h)_{L^2(\Omega)}
	&= -(\nabla u_h,\nabla w_h)_{L^2(\Omega)} + \int_\Gamma \nabla u_h\cdot n_\Gamma\,w_h\,\ds \\
	&= (\Delta_h u_h, w_h)_{L^2(\Omega)} - \sum_{F\in\InnerEdges} \int_F \jump{\nabla u_h\cdot n_F}\,w_h\,\ds,
\end{align*}
where we exploited that $w_h$ is continuous across the element facets. We choose the test function $w_h = \FELaplaceCG{u_h}$, use the orthogonality of the $L^2(\Omega)$-projection $\P_{\HSpaceCG}$ onto $\HSpaceCG$ and get with a further application of the definition \eqref{eq:def_discrete_Hessian} and the integration-by-parts formula
\begin{align*}
	\norm{\FELaplaceCG{u_h}}_{L^2(\Omega)}^2 &=
	(\P_{\HSpaceCG} (\Delta_h u_h), \FELaplaceCG{u_h})_{L^2(\Omega)}
	- \sum_{F\in\InnerEdges} \int_F \jump{\nabla u_h\cdot n_F}\,\FELaplaceCG{u_h}\,\ds \\
	&= - (\nabla \P_{\HSpaceCG}(\Delta_h u_h), \nabla u_h)_{L^2(\Omega)}
	- \sum_{F\in\InnerEdges} \int_F \jump{\nabla u_h\cdot n_F}\,\FELaplaceCG{u_h}\,\ds \\
	&\quad + \int_\Gamma \nabla u_h\cdot n_\Gamma\,\P_{\HSpaceCG} (\Delta_h u_h)\,\ds \\
	&= (\P_{\HSpaceCG}(\Delta_h u_h), \Delta_h u_h)_{L^2(\Omega)} - \sum_{F\in\InnerEdges} \jump{\nabla u_h\cdot n_F}\,(\FELaplaceCG{u_h} + \P_{\HSpaceCG}(\Delta_h u_h))\,\ds.
\end{align*}
From this we infer with the properties of $\P_{\HSpaceCG}$ and similar arguments as in \eqref{eq:Proof_Bound_HLaplace_CG}
\begin{equation*}
	\norm{\FELaplaceCG{u_h}}_{L^2(\Omega)}^2
	\ge \frac13 \norm{\P_{\HSpaceCG}(\Delta_h u_h)}_{L^2(\Omega)}^2 - \frac{d+1}3 \Ctr^2\sum_{F\in\InnerEdges} h_F^{-1}\norm{\jump{\nabla u_h\cdot n_F}}_{L^2(F)}^2.
\end{equation*}
Furthermore, with the property $\norm{\P_{\HSpaceCG} v}_{L^2(\Omega)}^2 = \norm{v}_{L^2(\Omega)}^2 + \norm{v-\P_{\HSpaceCG} v}_{L^2(\Omega)}^2$ and the estimate \eqref{eq:CG_lifting_error_estimate} we get
\begin{align*}
	\norm{\FELaplaceCG{u_h}}_{L^2(\Omega)}^2
	\ge \frac13 \norm{\Delta_h u_h}_{L^2(\Omega)}^2 &- \frac{d+1}3 \Ctr^2\sum_{F\in\InnerEdges} h_F^{-1}\norm{\jump{\nabla u_h\cdot n_F}}_{L^2(F)}^2 \\
	&- \frac{d+1}6 \CE^2\,\sum_{F\in \InnerEdges} h_F\, \norm{\jump{\Delta u_h}}_{L^2(F)}^2.
\end{align*}
\end{proof}

Next, we want to mimic the proof of \cref{lem:existence} for the continuous setting in order
to show well-posedness of our discrete scheme. 
However, due to the jump terms on the right-hand sides of the estimates
\eqref{eq:Discrete_Miranda_Talenti} and \eqref{eq:Discrete_Miranda_TalentiDG} the proof of the coercivity 
of the bilinear form $a_h$ will fail if $\varepsilon$ is too small,
see \eqref{eq:CordesCondition}. To this end, stabilization terms in the discrete scheme are needed and we define
\begin{align}
	\label{eq:Discrete_bilinear_form_stab}
	J_h(u_h,v_h)
	&\coloneqq \eta_1\,\sum_{F\in\InnerEdges} h_F^{-1}\int_F \jump{\nabla u_h\cdot n_F}\,\jump{\nabla v_h\cdot n_F}\,\ds
	+ \eta_2\,\sum_{F\in\InnerEdges} h_F \int_F \jump{\nabla^2 u_h}\cdot \jump{\nabla^2 u_h}\,\ds
\end{align}
to be inserted into \eqref{eq:DiscreteProblem}. The penalty parameters $\eta_1,\eta_2\ge 0$ have to be chosen appropriately to guarantee the coercivity of $a_h+J_h$.
For the stabilized scheme one can show the following well-posedness result.
\begin{lemma}
	\label{lem:Lax_Milgram}
	Let $\Omega\subset\R^d$ be a bounded and convex domain.
	Let $C_1$ and $C_2$ be the constants from \cref{lem:Discrete_Miranda_Talenti_with_Jumps},
	where we set $C_2=0$ in case of $\FEHessian{} =
	\FEHessianDG{}$.
	Assume that $A$ fulfills the \Cordes\ condition \eqref{eq:CordesCondition} 
	with a constant $\varepsilon\in (0,1]$ and that the polynomial degree of $\uSpace$ and $\HSpace$ is $p\ge 2$.
	Then the bilinear form $a_h+J_h$ is bounded and uniformly elliptic, \ie,
	there exist constants $\alpha_0,\beta_0$ such that the inequalities
	\begin{align}
		a_h(u_h,v_h) + J_h(u_h,v_h) &\le \beta_0\,\Hnorm{u_h}\, \Hnorm{v_h} 
		&& \forall u_h, v_h\in \uSpaceZero,\label{eq:ah_bounded} \\
		a_h(u_h,u_h) + J_h(u_h,u_h) &\ge \constCordes\,\alpha_0\,\Hnorm{u_h}^2
		&&\forall u_h\in \uSpaceZero,\label{eq:discrete_ellipticity}
	\end{align}
	are fulfilled, provided that
	the penalty parameters in $J_h$ fulfill the inequalities
	
	\begin{equation}
		\label{eq:assumptions_penalty_parameters}
		\eta_1 \ge \frac{(1-\varepsilon)\,C_1}2+ \frac{\constCordes}6\,\varh(){1 + (d+1)\,\Ctr^2}
		\qquad
		\eta_2 \ge \frac{(1-\varepsilon)\,C_2}2 + \frac{d+1}{12}\,\constCordes\,\CE^2.
	\end{equation}
	The constants $\alpha_0$ and $\beta_0$ depend on $\kappa, d, p, \Ctr, \Cv$ and in the case $d=3$ on the geometry of $\Omega$, but not on $\constCordes$ and $h$.
	
	As a consequence, problem \eqref{eq:DiscreteProblem} possesses a unique solution $u_h\in \uSpaceZero$ for each $f\in L^2(\Omega)$.
\end{lemma}
\begin{proof}
	First, we show the boundedness of $a_h+J_h$. With the Cauchy-Schwarz inequality and
	\eqref{eq:stability_hessian} we obtain
	\begin{equation}
	\label{eq:bilinear_form_bounded_by_norm}
		a_h(u_h,v_h)
		\le c_d\,\Ctr^2\,\norm{\gamma A}_{L^\infty(\Omega)} \Hnorm{u_h}\,\Hnorm{v_h}
		.
	\end{equation}
	With \cref{lem:consequence_cordes} and the assumed \Cordes\ condition \eqref{eq:CordesCondition} we moreover conclude $\norm{\gamma\,A}_{L^\infty(\Omega)} \le 1+\sqrt{1-\varepsilon} \le 2$.
	To derive a similar estimate for the stabilization term $J_h$ we apply the Cauchy-Schwarz inequality on each inner facet $F\in \InnerEdges$, and for the second term in $J_h$ we additionally employ the discrete trace theorem from~\cref{lem:invserse_trace} $\norm{\nabla^2 u_h}_{L^2(F)} \le \Ctr\,h_{T_F}^{-1/2}\,\norm{\nabla^2 u_h}_{L^2(T_F)}$,
	where $T_F\in \TT_h$ is an arbitrary element with $F\subset T_F$.
	This yields
	\begin{equation}\label{eq:jump_bounded_by_norm_no_penalty_parameters}
	\begin{split}
	\sum_{F\in\InnerEdges} h_F^{-1}\int_F \jump{\nabla u_h\cdot n_F}\,\jump{\nabla v_h\cdot n_F}\,\ds	
	&\le \Hnorm{u_h}\,\Hnorm{v_h},\\
	\sum_{F\in\InnerEdges} h_F \int_F \jump{\nabla^2 u_h}\cdot \jump{\nabla^2 v_h}\,\ds
	&\le \Ctr^2\,(d+1)\,\Hnorm{u_h}\,\Hnorm{v_h}.
	\end{split}
	\end{equation}
	from which we deduce
	\begin{equation}
		\label{eq:jump_bounded_by_norm}
		J(u_h,v_h)
		\le
		\varh(){\eta_1 + \eta_2\,\Ctr^2\,(d+1)}\,\Hnorm{u_h}\,\Hnorm{v_h}.
	\end{equation}
	The inequalities \eqref{eq:bilinear_form_bounded_by_norm} and \eqref{eq:jump_bounded_by_norm} lead to \eqref{eq:ah_bounded}.

	The coercivity follows from \cref{lem:Discrete_Miranda_Talenti_with_Jumps},
	taking into account the \Cordes\ condition \eqref{eq:CordesCondition} with
	the estimate from \cref{lem:consequence_cordes} and Young's inequality with weight $\xi>0$, \ie,
	\begin{align}
		\label{eq:discrete_coercivity_proof}
		&\int_\Omega \gamma A\dprod\FEHessian{u_h}\,\FELaplace{u_h}\dx
		= \norm{\FELaplace{u_h}}_{L^2(\Omega)}^2 + \int_\Omega(\gamma A - I) \dprod \FEHessian{u_h}\,\FELaplace{u_h}\dx
		\nonumber\\
		&\qquad\ge \norm{\FELaplace{u_h}}_{L^2(\Omega)}^2
		- \sqrt{1-\constCordes}\,\norm{\FELaplace{u_h}}_{L^2(\Omega)}\nonumber\\
		&\qquad\quad\times\Biggh(){\norm{\FELaplace{u_h}}_{L^2(\Omega)}^2 + C_1 \sum_{F\in\InnerEdges} h_F^{-1} \norm{\jump{\nabla u_h\cdot n_F}}_{L^2(F)}^2
		+ C_2 \sum_{F\in\InnerEdges} h_F \norm{\jump{\nabla^2 u_h}}_{L^2(F)}^2}^{1/2} \nonumber\\
		&\qquad \ge \varh(){1- \frac{\xi \, (1-\constCordes) + \xi^{-1}}2} \norm{\FELaplace{u_h}}_{L^2(\Omega)}^2 \nonumber\\
		&\qquad\quad - \frac{1}{2\xi} \Biggh(){C_1 \sum_{F\in\InnerEdges} h_F^{-1} \norm{\jump{\nabla u_h\cdot n_F}}_{L^2(F)}^2
		+ C_2 \sum_{F\in\InnerEdges} h_F\,\norm{\jump{\nabla^2 u_h}}_{L^2(F)}^2}.
	\end{align}
	The jump terms in \eqref{eq:discrete_coercivity_proof} can be canceled by the stabilization
	terms from \eqref{eq:Discrete_bilinear_form_stab}.
	We use the choice $\xi = (1-\constCordes)^{-1}$ and insert the estimate from \cref{eq:Bound_Laplace_FELaplace} to arrive at
	\begin{align}
	\label{eq:ellipticity_proof}
		&a_h(u_h,u_h) + J_h(u_h,u_h)
		\ge \frac{\constCordes}2\, \norm{\FELaplace{u_h}}_{L^2(\Omega)}^2 
		+ \varh(){\eta_1 - \frac{(1-\varepsilon)\,C_1}2} \sum_{F\in\InnerEdges} h_F^{-1} \norm{\jump{\nabla u_h\cdot n_F}}_{L^2(F)}^2 \nonumber\\
		&\qquad + \varh(){\eta_2 - \frac{(1-\varepsilon)\,C_2}2} \sum_{F\in\InnerEdges} h_F \norm{\jump{\nabla^2 u_h}}_{L^2(F)}^2 \nonumber\\
		&\quad\ge \frac{\constCordes}6 \norm{\Delta_h u_h}_{L^2(\Omega)}^2 + \varh(){\eta_1 - \frac{(1-\varepsilon)\,C_1}2- \frac{d+1}6\,\constCordes\,\Ctr^2 } \sum_{F\in\InnerEdges} h_F^{-1} \norm{\jump{\nabla u_h\cdot n_F}}_{L^2(F)}^2 \nonumber\\
		&\qquad + \varh(){\eta_2 - \frac{(1-\varepsilon)\,C_2}2 - \frac{d+1}{12}\,\constCordes\,\CE^2} \sum_{F\in\InnerEdges} h_F \norm{\jump{\nabla^2 u_h}}_{L^2(F)}^2
		.
	\end{align}
	Taking into account the assumptions \eqref{eq:assumptions_penalty_parameters} we may further estimate
	\begin{equation*}
		a_h(u_h,u_h) + J_h(u_h,u_h) \ge \frac\constCordes6\,\varh(){\norm{\Delta_h u_h}_{L^2(\Omega)}^2 + \sum_{F\in\InnerEdges} h_F^{-1}\,\norm{\jump{\nabla u_h\cdot n_F}}_{L^2(F)}^2}.
	\end{equation*}
	The right-hand side forms a norm on $H_h^2(\Omega)\cap H_0^1(\Omega)$ which is equivalent to the norm defined in \eqref{eq:def_Hh2_norm}. This is a consequence of a Miranda-Talenti estimate for the broken Hessian, see \cite{NeilanWu2019}.

	The Lax-Milgram Lemma finally implies the existence and uniqueness of a discrete solution $u_h\in \uSpaceZero$ of \eqref{eq:DiscreteProblem}.
\end{proof}

\begin{remark}
	\label{remark:assumptions_penalty_parameters}
	The assumption \eqref{eq:assumptions_penalty_parameters} can be relaxed such that the choice $\eta_1=\eta_2=0$ is also feasible. This requires the following modification in the proof of the previous theorem. As $\norm{\FELaplace{\cdot}}_{L^2(\Omega)}$ is also a norm in the finite-dimensional space $\uSpaceZero$, there exists a constant $\CH$ independent of $\constCordes$ such that the estimate $\norm{\FELaplace{u_h}}_{L^2(\Omega)}^2 \ge	\CH\,\Hnorm{u_h}^2$ is valid for all $u_h\in \uSpaceZero$. Using this estimate and \eqref{eq:jump_bounded_by_norm_no_penalty_parameters} we can modify the last step in \eqref{eq:ellipticity_proof} to arrive at
	\begin{align*}
		&a_h(u_h,u_h) + J_h(u_h,u_h)
		\ge \frac{\constCordes}2\, \norm{\FELaplace{u_h}}_{L^2(\Omega)}^2 
		+ \varh(){\eta_1 - \frac{(1-\varepsilon)\,C_1}2} \sum_{F\in\InnerEdges} h_F^{-1}\,\norm{\jump{\nabla u_h\cdot n_F}}_{L^2(F)}^2 \\
		&\qquad + \varh(){\eta_2 - \frac{(1-\varepsilon)\,C_2}2} \sum_{F\in\InnerEdges} h_F\, \norm{\jump{\nabla^2 u_h}}_{L^2(F)}^2 \\
		&\ge \frac{\constCordes}4\, \norm{\FELaplace{u_h}}_{L^2(\Omega)}^2
		+ \varh(){\eta_1- \frac{(1-\varepsilon)\,C_1}2+\frac{\constCordes\,\CH}{8}}\,\sum_{F\in\InnerEdges} h_F^{-1}\, \norm{\jump{\nabla u_h\cdot n_F}}_{L^2(F)}^2 \\
		&\qquad + \varh(){\eta_2 - \frac{(1-\varepsilon)\,C_2}2+ \frac{\constCordes\,\CH}{8\,\Ctr^2\,(d+1)}} \sum_{F\in\InnerEdges} h_F\, \norm{\jump{\nabla^2 u_h}}_{L^2(F)}^2.
	\end{align*}	
	One observes that coercivity of $a_h+J_h$ can be guaranteed without the presence of the penalty terms, provided that $4\,(1-\constCordes)\,C_1 \le \constCordes\,\CH$ and $4\,\Ctr\,(d+1)\,(1-\constCordes)\,C_2 \le \constCordes\,\CH$ are fulfilled. Note that $C_1,C_2,\Ctr$ and $\CH$ are independent of $\constCordes$. Thus, these inequalities are valid when $\varepsilon$ is sufficiently close to $1$. 
	
    In the numerical experiments we observed that neglecting the penalty terms $J_h$ is in most situations feasible, but has negative influence on the robustness of preconditioned iterative solvers. However, the experimental convergence rates are better when the penalty terms are omitted, see \cref{sec:cordes_violated}.
\end{remark}

\subsection{A~priori and a~posteriori error estimates}
\label{sec:ErrorEstimatesFEHessian}

This section is devoted to the a~priori and a~posteriori error analysis of the finite
element approximation \eqref{eq:DiscreteProblem}.

\begin{theorem}[A priori error estimate]
	\label{thm:a_priori_estimate}
	Let $\Omega\subset\R^d$ be a bounded and convex domain.
	Assume that the penalty parameters $\eta_1,\eta_2$ in $J_h$ satisfy
	\eqref{eq:assumptions_penalty_parameters} and that the solution $u$ of
	\eqref{eq:MainProblem} belongs to $H^s(\Omega)$ with
	\begin{equation}
		\label{eq:regularity_assumption}
		s\le p+1\ \text{and}\ 
		s {} > {} \begin{cases}
			5/2, &\text{if}\ \FEHessian{}=\FEHessianCG{},\\
			2, &\text{if}\ \FEHessian{}=\FEHessianDG{}.
		\end{cases}
	\end{equation}
	The approximate solutions $u_h\in\uSpaceZero$ of \eqref{eq:DiscreteProblem} with $p\ge 2$ fulfill the a~priori error estimate
	\begin{equation*}
		\Hnorm{u-u_h} \le c_{\kappa,p,d}\,\varh(){1+\frac{\Ctr(1+\eta_1+\eta_2)}{\varepsilon\,\alpha_0}}\,h^{s-2}\,\norm{u}_{H^s(\Omega)}.
	\end{equation*}
\end{theorem}
\begin{proof}
	We introduce the nodal interpolant $\I_{\uSpace}(u)$ as an intermediate function and deduce with \cref{lem:H2h_norm_bound} and standard interpolation error estimates 
	\begin{align}\label{eq:interpolation_error}
		\Hnorm{u-\I_{\uSpace}(u)}
		&\le c_{\kappa,d}\left(h^{-1}\,\norm{\nabla(u-\I_{\uSpace}(u))}_{L^2(\Omega)} + \norm{\nabla_h^2(u-\I_{\uSpace}(u))}_{L^2(\Omega)}\right)\nonumber\\
		&\le c_{\kappa,p,d}\,h^{s-2}\,\abs{u}_{H^s(\Omega)}.
	\end{align}

	Next, we derive an estimate for the norm of the discrete function $w_h \coloneqq u_h-\I_{\uSpace}(u)$.
	Therefore, we apply the discrete ellipticity \eqref{eq:discrete_ellipticity}, the definition
	of $u_h$ and the strong formulation \eqref{eq:strong_form_L} taking into
	account $\FELaplace{w_h}\in \HSpace \subset L^2(\Omega)$ as well as $J_h(u,w_h)=0$ which holds under the assumption \eqref{eq:regularity_assumption}. These arguments imply
	\begin{align}\label{eq:a_priori_fully_discrete}
		\constCordes\,\alpha_0 \, \Hnorm{u_h - \I_{\uSpace}(u)}^2
		&\le a_h(u_h-\I_{\uSpace}(u), w_h) + J_h(u_h-\I_{\uSpace}(u), w_h) \nonumber\\
		&= \int_\Omega\gamma f\,\FELaplace{w_h}\,\dx - a_h(\I_{\uSpace}(u), w_h)
		- J_h(\I_{\uSpace}(u), w_h) \nonumber\\
		&= \int_\Omega\gamma A \dprod (\nabla^2 u - \FEHessian{\I_{\uSpace}(u)})\,\FELaplace{w_h}\,\dx + J_h(u-\I_{\uSpace}(u),w_h).
	\end{align}
	With the triangle inequality, \cref{lem:best_approximation_Hessian,lem:H2h_norm_bound} and standard interpolation error estimates we conclude
	\begin{align*}
		&\int_\Omega\gamma A \dprod \varh(){\nabla^2 u - \FEHessian{\I_{\uSpace}(u)}}\, \FELaplace{w_h}\,\dx \\
		&\qquad \le \norm{\gamma A}_{L^\infty(\Omega)}\,\varh(){\norm{\nabla^2 u - \FEHessian{u}}_{L^2(\Omega)} + \norm{\FEHessian{u-\I_{\uSpace}(u)}}_{L^2(\Omega)}} \,\norm{\FELaplace{w_h}}_{L^2(\Omega)} \\
		&\qquad \le c_{\kappa,p,d}\,\Ctr^2\,h^{s-2}\,\norm{\gamma A}_{L^\infty(\Omega)}\,\abs{u}_{H^s(\Omega)}\,\Hnorm{w_h}.
	\end{align*}
	Note again that $\norm{\gamma\,A}_{L^\infty(\Omega)}\le 2$ holds due to \eqref{eq:CordesCondition}. Analogously, we derive the following estimate for the jump terms in \eqref{eq:a_priori_fully_discrete}
	\begin{align*}
		 J_h(u-\I_{\uSpace}(u),w_h) 
		&\le \Ctr\,\Bigg(\sum_{F\in \InnerEdges} \Big(
		\eta_1^2\,h_F^{-1}\,\norm{\jump{\nabla(u-\I_{\uSpace}(u))\cdot n_F}}_{L^2(F)}^2 \\
		&\phantom{\le \Ctr\,\Bigg(\sum_{F\in \InnerEdges}}
		+ \eta_2^2\,h_F\,\norm{\jump{\nabla^2(u-\I_{\uSpace}(u))}}_{L^2(F)}^2\Big)\Bigg)^{1/2}\,\Hnorm{w_h} \\
		&\le c_{\kappa,p,d}\,\Ctr\,(\eta_1+\eta_2)\,h^{s-2}\,\norm{u}_{H^s(\Omega)}\,\Hnorm{w_h}.
	\end{align*}
	The latter step follows from the trace theorem $\norm{v}_{H^2(\partial \widehat T)}\le c_\kappa\,\norm{v}_{H^s(\widehat T)}$ on the reference element $\widehat T$ and the polynomial approximation results in fractional-order Sobolev spaces from \cite{DupontScott1980}.	
	After insertion of the previous two estimates into \eqref{eq:a_priori_fully_discrete}
	we arrive, together with \eqref{eq:interpolation_error}, at the assertion.
\end{proof}

\begin{theorem}[A posteriori error estimate]\label{thm:estimator_reliable}
	Under the assumptions of \cref{thm:a_priori_estimate}
	the solutions $u_h$ of \eqref{eq:DiscreteProblem} fulfill the a~posteriori error estimate
	\begin{equation}
	\label{eq:a_posteriori_estimate}
		\norm{u-u_h}_{H^2_h(\Omega)}^2 \le
		\sum_{T\in\TT_h} \Biggh(){
		 c_1\,\norm{\gamma f - \gamma A\dprod \nabla^2 u_h}_{L^2(T)}^2
		+c_2\,\sum_{F\in\FF_T\cap\InnerEdges} h_F^{-1}\norm{\jump{\nabla u_h\cdot n_F}}_{L^2(F)}^2
		},
	\end{equation}
	with $c_1 \coloneqq 4\,C_a^2$ and $c_2 \coloneqq \frac12(1+(2+16\,C_a^2)\,\CE^2)$,
	where $\FF_T$ denotes the set of facets of the element $T\in\TT_h$.
	The constants $C_a$ and $\CE$ are defined in \cref{lem:existence} and \eqref{eq:H2_lifting_estimate}.
\end{theorem}
\begin{proof}
	As in the proof of \cref{lem:Discrete_Miranda_Talenti_with_Jumps} we introduce
	the lifting operator $\LiftHCT$
	which maps functions from $\uSpace$ into the $H^2(\Omega)$-conforming
	HCT or virtual finite element space $\uSpaceConf$.
	With this operator at hand we introduce a further
	approximation of the finite element solution $u_h$, namely $\uconf = \LiftHCT(u_h)\in \uSpaceConf$.

	With the triangle inequality, the definition of the norm in $H_h^2(\Omega)\cap H_0^1(\Omega)$
	and the fact that the jump terms vanish for $u\in H^2(\Omega)$
	we may represent the error term under consideration by 
	\begin{align}\label{eq:a_posteriori_norm_splitting}
		\norm{u-u_h}_{H^2_h(\Omega)}^2
		&\le 2\,\Biggh(){\norm{\nabla^2 (u - \uconf)}_{L^2(\Omega)}^2
		+ \sum_{T\in\TT_h} \norm{\nabla^2 (\uconf - u_h)}_{L^2(T)}^2}\nonumber\\
		&\qquad + \sum_{F\in\InnerEdges} h_F^{-1} \norm{ \jump{\nabla u_h\cdot n_F}}_{L^2(F)}^2.
	\end{align}
	
	We start by proving an estimate for the first term
	on the right-hand side of \eqref{eq:a_posteriori_norm_splitting}. 
	We define the error functional $J\in X'$ (recall that $X=H^2(\Omega)\cap H_0^1(\Omega)$)
	\begin{equation}\label{eq:def_J}
		J(\varphi) \coloneqq \frac{\int_\Omega \nabla^2 (u-\uconf)\dprod\nabla^2 \varphi\,\dx}{\norm{\nabla^2 (u-\uconf)}_{L^2(\Omega)}}
	\end{equation}
	and easily confirm
	\begin{equation*}
		\norm{J}_{X'} = \sup_{\genfrac{}{}{0pt}{}{\varphi\in X}{\varphi\ne 0}} \frac{\abs{J(\varphi)}}{\norm{\varphi}_X}
		\le 1.
	\end{equation*}
	This functional forms the right-hand side of a dual equation
	\begin{equation*}
		a(\varphi, z) = J(\varphi) \quad \forall \varphi\in X
	\end{equation*}
	and from \cref{lem:existence} we conclude the existence of a unique
	solution $z\in X$ satisfying
	\begin{equation}\label{eq:a_priori_z}
		\norm{z}_{H^2(\Omega)} \le C_a\,\norm{J}_{X'} \le C_a.
	\end{equation}

	The definition of the lifting operator $\LiftHCT$ guarantees $u-\uconf\in X$ and thus,
	\begin{equation*}
		\norm{\nabla^2(u-\uconf)}_{L^2(\Omega)}
		= J(u-\uconf) = a(u-\uconf, z).
	\end{equation*}
	The right-hand side of the previous equation is treated as follows. We apply \eqref{eq:weak_form}, insert the intermediate function $\gamma\,A\dprod\nabla_h^2 u_h$, apply the Cauchy-Schwarz inequality as well as \eqref{eq:a_priori_z} to obtain
	\begin{align}\label{eq:a_posteriori_splitting}
		&\norm{\nabla^2(u-\uconf)}_{L^2(\Omega)} = a(u-\widetilde u_h, z)\nonumber\\
		&\quad = \int_\Omega \left(\gamma f - \gamma A \dprod \nabla^2 \uconf\right)\Delta z\,\dx\nonumber\\ 
		&\quad = \sum_{T\in\TT_h}\int_T \left((\gamma f- \gamma A \dprod \nabla^2 u_h) + \gamma A\dprod \nabla^2(u_h - \uconf)\right)\Delta z\,\dx \nonumber\\
		&\quad\le C_a\,\sqrt{2}\,\Biggh(){\sum_{T\in\TT_h} \Bigh(){\norm{\gamma f-\gamma A\dprod \nabla^2 u_h}_{L^2(T)}^2 + \norm{\gamma A}_{L^\infty(\Omega)}^2 \norm{\nabla^2(u_h - \uconf)}_{L^2(T)}^2}}^{1/2}.
	\end{align}
	Finally, using $\norm{\gamma\,A}_{L^\infty(\Omega)}\le 2$, insertion of \eqref{eq:a_posteriori_splitting} 
	into \eqref{eq:a_posteriori_norm_splitting} and applying the estimate
	\eqref{eq:lifting_error_estimate} for the lifting error terms leads to the desired result.
\end{proof}
The error estimate from the previous lemma provides a local a~posteriori error estimator, namely
\begin{equation}\label{eq:local_estimator}
	\eta_T^2(u_h)\coloneqq \norm{\gamma f - \gamma A\dprod \nabla^2 u_h}_{L^2(T)}^2 +
	\sum_{F\in\FF_T\cap\InnerEdges} h_F^{-1}\norm{\jump{\nabla u_h\cdot n_F}}_{L^2(F)}^2,
\end{equation}
and a global estimator
\begin{equation}\label{eq:global_estimator}
	\eta^2(u_h)\coloneqq\sum_{T\in\TT_h} \eta_T^2(u_h)
\end{equation}
which is a reliable bound for the error $\Hnorm{u-u_h}$.

\begin{theorem}
	Let the assumptions of \cref{thm:a_priori_estimate} be fulfilled.
	The a~posteriori error estimate \eqref{eq:a_posteriori_estimate} is sharp in the sense that
	\begin{equation*}
		\frac12\,\eta_T^2(u_h)\le \norm{u-u_h}_{H^2_h(T)} \coloneqq \Biggh(){\norm{\nabla^2(u-u_h)}_{L^2(T)}^2 + \sum_{F\in \FF_T\cap\InnerEdges}h_F^{-1}\norm{\jump{\nabla (u-u_h)\cdot n_F}}_{L^2(F)}^2}^{1/2}.
	\end{equation*}
\end{theorem}
\begin{proof}
	The jump terms from the left-hand side of the desired estimate appear also in the
	norm of the right-hand side.
	We merely have to take into account that $\jump{\nabla u\cdot n_F} = 0$
	a.e.\ on all interior facets $F\in\InnerEdges$.
	The volume residuals are bounded by the element-wise $H^2(\Omega)$-seminorm due to
	\begin{equation*}
		\norm{\gamma f-\gamma A \dprod \nabla^2 u_h}_{L^2(T)}
		= \norm{\gamma A \dprod (\nabla^2 u - \nabla^2 u_h)}_{L^2(T)}
		\le \norm{\gamma A}_{L^\infty(T)}\,\norm{\nabla^2(u-u_h)}_{L^2(T)}
	\end{equation*}
	and $\norm{\gamma\,A}_{L^\infty(\Omega)}\le 2$.
\end{proof}

\subsection{A method using the piecewise Hessian}
\label{sec:PiecewiseHessianApproach}

Instead of using Hessian recovery techniques for the realization of our method, as investigated in the previous sections, it is also possible to use the cellwise exact Hessian, \ie, $\FEHessian{} \coloneqq \nabla_h^2$. This idea is proposed in \cite{NeilanSalgadoZhang2017}. As the resulting bilinear form is not coercive additional jump penalty terms have to be added. The resulting equation reads
\begin{multline}\label{eq:direct_scheme}
		\text{Find}\ u_h\in \uSpaceZero \text{ s.t.}
		\\
		a_h(u_h,v_h) \coloneqq \int_\Omega \gamma A\colon \nabla_h^2 u_h\,\Delta_h v_h\,\dx
		+ \eta_1\,\sum_{F\in \InnerEdges}h_F^{-1} \int_F \jump{\nabla u_h\cdot n_F}\,\jump{\nabla v_h\cdot n_F}\,\ds
		= \int_\Omega \gamma f\,\Delta v_h\,\dx
\end{multline}
for all $v_h\in \uSpaceZero$. Under the assumption that $\eta_1>0$ is sufficiently large ($\eta_1=0$ is not allowed here) and that the \Cordes\ condition \eqref{eq:CordesCondition} is fulfilled with some $\varepsilon\in (0,1]$, it has been  proved in \cite[Lemma~4.3]{NeilanSalgadoZhang2017} that the bilinear form $a_h(\cdot,\cdot)$ is uniformly coercive on $\uSpaceZero$ and hence, \eqref{eq:direct_scheme} possesses a unique solution $u_h\in \uSpaceZero$. This is a direct consequence of a discrete Miranda-Talenti estimate similar to \cref{lem:Discrete_Miranda_Talenti_with_Jumps} and the techniques applied in the proof of \cref{lem:Lax_Milgram}.

Due to the consistency of this scheme, one can easily conclude the a~priori estimate
\begin{equation*}
	\Hnorm{u-u_h} \le C\,h^{s-2} \abs{u}_{H^s(\Omega)},\quad s\in [2,p+1],
\end{equation*}
provided that $u$ belongs to $H^s(\Omega)$.

A~posteriori error estimates can be derived with the same argument as in \cref{thm:estimator_reliable}. To be more precise, one can show by a slight modification of the proofs from the previous section that the estimator from \eqref{eq:global_estimator} is a reliable and sharp bound for $\Hnorm{u-u_h}$.

An advantage of the direct scheme \eqref{eq:direct_scheme} is that the computational effort is less than for our system \eqref{eq:DiscreteProblem}
since no additional equations for the computation of the Hessian approximation are needed.
As we will observe in our numerical experiments, the approximation properties for the error $u-u_h$ in the $H^2_h(\Omega)$-norm as well as in the $H^1(\Omega)$-norm will be the same for both approaches.
However, it turns out that the convergence rate in the $L^2(\Omega)$-norm is higher for
the approach studied in the previous sections.

\section{Numerical experiments}
\label{sec:experiments}

In this section, we perform different numerical experiments.
All implementations were done in \python\ using the finite element library \fenics~2019.1 \cite{AlnaesBlechtaHakeJohanssonKehletLoggRichardsonRingRognesWells2015,LoggWellsHake2012:1}.
\textbf{Our code is residing in a GitHub repository and it will be made publicly available upon acceptance of the manuscript.}

It is our purpose to compare four discretization approaches, \ie,
\begin{itemize}
	\item 
		the method using a finite-element Hessian with continuous and discontinuous trial functions (denoted by CG and DG in the following) discussed in the present article (\cref{sec:discretization_approach}), 
	\item
		the Petrov-Galerkin scheme (N) proposed by Neilan \cite{Neilan2017}, which likewise utilizes a DG finite-element Hessian but with $\tau_h = \text{id}$, \ie, there is no Laplacian acting on the test function,
	\item
		and the method using the piecewise Hessian proposed by Neilan, Salgado and Zhang (NSZ) \cite{NeilanSalgadoZhang2017} that we discussed briefly in \cref{sec:PiecewiseHessianApproach}.
\end{itemize}

\subsection{A problem with almost violated \Cordes\ condition}
\label{sec:cordes_violated}

\begin{table}[bht]
	\begin{center}
		\begin{tabular}{ccc@{\hskip 1cm}cc@{\hskip 1cm}cc}
			\toprule
			& \multicolumn{2}{c}{$\kappa=0.9$\quad\phantom{a}} & \multicolumn{2}{c}{$\kappa=0.99$\quad\phantom{a}} & \multicolumn{2}{c}{$\kappa=0.999$}\\
			$h$& $\eta_1=0$ & $\eta_1=1$ & $\eta_1=0$ & $\eta_1=1$ & $\eta_1=0$ & $\eta_1=1$\\ \midrule
            $2^{-3}$& 23&  13  &  25 & 13   & 26  &  13\\
            $2^{-4}$& 27&  17  &  25 & 17   & 25  &  17\\
            $2^{-5}$& 23&  19  &  25 & 19   & 30  &  19\\
            $2^{-6}$& 24&  19  &  26 & 19   & 27  &  19\\
            $2^{-7}$& 23&  19  &  25 & 20   & 27  &  20\\
            $2^{-8}$& 23&  20  &  25 & 27   & 26  &  20\\
			\bottomrule
		\end{tabular}
	\end{center}
    \caption{Iteration number for \gmres\ to achieve an absolute and relative
    tolerance of $10^{-8}$ for the method using a finite element Hessian with continuous trial functions, varying $\eta_1$ and fixed $\eta_2 = 0$.}
	\label{tab:iteration_numbers_nocordes}
\end{table}

We choose a problem on the unit square with matrix
\begin{equation*}
	A = \begin{pmatrix} 1 & \kappa \\ \kappa & 1 \end{pmatrix}
\end{equation*}
and determine the source term $f$ such that the smooth, exact solution of \eqref{eq:MainProblem} is given by
\begin{equation*}
	u(x) = \sin(2 \, \pi \, x_1) \, \sin(2 \, \pi \, x_2).
\end{equation*}
The matrix~$A$ fulfills the \Cordes\ condition if $\kappa \in(-1,1)$. If
$\kappa$ is sent to $1$, $\constCordes$ and hence the coercivity constant from \cref{lem:Lax_Milgram} will tend to zero
so that the problem is harder to solve with an iterative method like \gmres.
This behavior is also observed in our numerical experiments. 
The iteration numbers required to realize our method with a CG Hessian for piecewise quadratic trial functions ($p=2$) for different stabilization parameters in $J_h$ and different values of $\kappa$ are reported in \cref{tab:iteration_numbers_nocordes}. Obviously, with the preconditioner proposed in \eqref{eq:preconditioner} and the stabilization term $J_h$, we observe that the iteration numbers mildly increase when the mesh parameter decreases or when $\kappa$ approaches~$1$. 
The incorporation of an additional jump term for the second derivatives in $J_h$, \ie, the choice $\eta_2 > 0$ in \eqref{eq:Discrete_bilinear_form_stab}, did not lead to an improvement of the computational results.

\pgfplotstableread[col sep = comma]{./results/Degree_Study/degree_study_qd100_Cinfty_kappa_05_deg_1.csv}\degOne%
\pgfplotstableread[col sep = comma]{./results/Degree_Study/degree_study_qd100_Cinfty_kappa_05_deg_2.csv}\degTwo%
\pgfplotstableread[col sep = comma]{./results/Degree_Study/degree_study_qd100_Cinfty_kappa_05_deg_3.csv}\degThree%
\pgfplotstableread[col sep = comma]{./results/Degree_Study/degree_study_qd100_Cinfty_kappa_05_deg_4.csv}\degFour%
\begin{figure}[bht]
	\begin{subfigure}{0.33\textwidth}
		\centering
		\resizebox{\textwidth}{!}{
			\begin{tikzpicture}
				\begin{loglogaxis}[enlargelimits=false,
						xlabel={$\dim(V_h)$},
						legend to name=exp1degreeStudyLegend,
					legend columns = 7]
					\addplot+[cgZeroStab] table[x=Ndofs,y=cg0stab_L2_error] {\degOne};
					\addlegendentry{$\text{CG}_{\eta_1=0,\eta_2=0}$};
					\addplot+[NeilanSalgadoZhang] table[x=Ndofs,y=nsz_L2_error] {\degOne};
					\addlegendentry{$\text{NSZ}$};
					\addlegendimage{solid, black};
					\addlegendentry{$p = 1$};
					\addlegendimage{dashed, black};
					\addlegendentry{$p = 2$};
					\addlegendimage{dash dot, black};
					\addlegendentry{$p = 3$};
					\addlegendimage{densely dotted, black};
					\addlegendentry{$p = 4$};
					\addplot+[cgZeroStab, dashed] table[x=Ndofs,y=cg0stab_L2_error] {\degTwo};
					\addplot+[NeilanSalgadoZhang, dashed] table[x=Ndofs,y=nsz_L2_error] {\degTwo};
					\addplot+[cgZeroStab, dash dot] table[x=Ndofs,y=cg0stab_L2_error] {\degThree};
					\addplot+[NeilanSalgadoZhang, dash dot] table[x=Ndofs,y=nsz_L2_error] {\degThree};
					\addplot+[cgZeroStab, densely dotted] table[x=Ndofs,y=cg0stab_L2_error] {\degFour};
					\addplot+[NeilanSalgadoZhang, densely dotted] table[x=Ndofs,y=nsz_L2_error] {\degFour};
				\end{loglogaxis}
			\end{tikzpicture}
		}
		\caption{$\norm{u-u_h}_{L^2(\Omega)}$}
	\end{subfigure}\hfill
	\begin{subfigure}{0.33\textwidth}
		\centering
		\resizebox{\textwidth}{!}{
			\begin{tikzpicture}
				\begin{loglogaxis}[enlargelimits=false,
					xlabel={$\dim(V_h)$}]
					\addplot+[cgZeroStab] table[x=Ndofs,y=cg0stab_H1_error] {\degOne};
					\addplot+[NeilanSalgadoZhang] table[x=Ndofs,y=nsz_H1_error] {\degOne};
					\addplot+[cgZeroStab, dashed] table[x=Ndofs,y=cg0stab_H1_error] {\degTwo};
					\addplot+[NeilanSalgadoZhang, dashed] table[x=Ndofs,y=nsz_H1_error] {\degTwo};
					\addplot+[cgZeroStab, dash dot] table[x=Ndofs,y=cg0stab_H1_error] {\degThree};
					\addplot+[NeilanSalgadoZhang, dash dot] table[x=Ndofs,y=nsz_H1_error] {\degThree};
					\addplot+[cgZeroStab, densely dotted] table[x=Ndofs,y=cg0stab_H1_error] {\degFour};
					\addplot+[NeilanSalgadoZhang, densely dotted] table[x=Ndofs,y=nsz_H1_error] {\degFour};
				\end{loglogaxis}
			\end{tikzpicture}
		}
		\caption{$\norm{u-u_h}_{H^1(\Omega)}$}
	\end{subfigure}\hfill
	\begin{subfigure}{0.33\textwidth}
		\centering
		\resizebox{\textwidth}{!}{
			\begin{tikzpicture}
				\begin{loglogaxis}[enlargelimits=false,
					xlabel={$\dim(V_h)$}]
					\addplot+[cgZeroStab] table[x=Ndofs,y=cg0stab_H2h_error] {\degOne};
					\addplot+[NeilanSalgadoZhang] table[x=Ndofs,y=nsz_H2h_error] {\degOne};
					\addplot+[cgZeroStab, dashed] table[x=Ndofs,y=cg0stab_H2h_error] {\degTwo};
					\addplot+[NeilanSalgadoZhang, dashed] table[x=Ndofs,y=nsz_H2h_error] {\degTwo};
					\addplot+[cgZeroStab, dash dot] table[x=Ndofs,y=cg0stab_H2h_error] {\degThree};
					\addplot+[NeilanSalgadoZhang, dash dot] table[x=Ndofs,y=nsz_H2h_error] {\degThree};
					\addplot+[cgZeroStab, densely dotted] table[x=Ndofs,y=cg0stab_H2h_error] {\degFour};
					\addplot+[NeilanSalgadoZhang, densely dotted] table[x=Ndofs,y=nsz_H2h_error] {\degFour};
				\end{loglogaxis}
			\end{tikzpicture}
		}
		\caption{$\norm{u-u_h}_{H^2_h(\Omega)}$}
	\end{subfigure}
	\bigskip

	\centering
	\hypersetup{hidelinks}
	\ref{exp1degreeStudyLegend} 
	\caption{Comparison of absolute errors for different polynomial degrees for the example from~\cref{sec:cordes_violated} with $\kappa = 0.5$.}
	\label{fig:comparison_methods}
\end{figure}
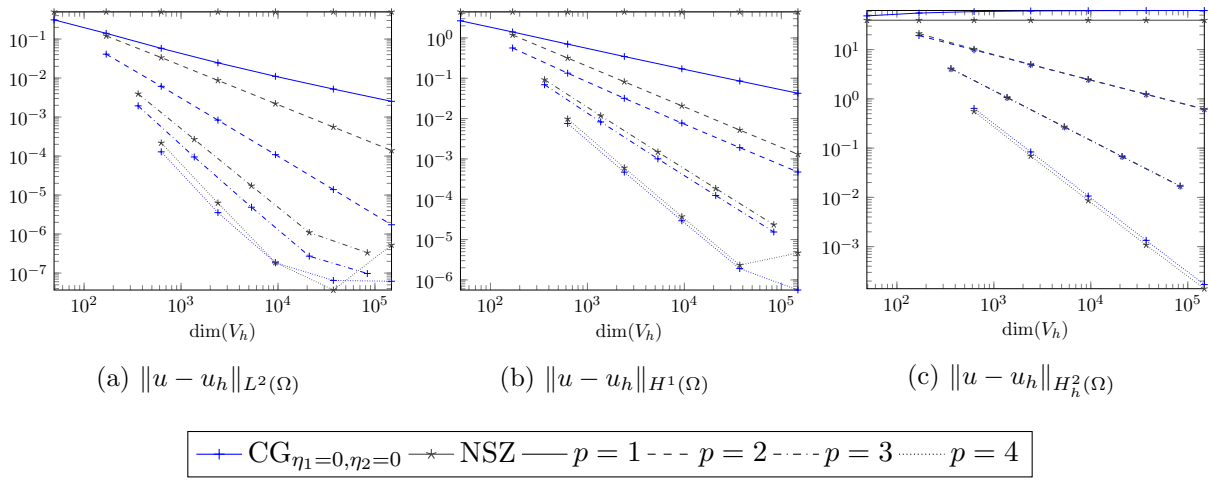

In a further numerical test, we computed the discretization error for
different polynomial degrees. Here, we used the choice $\kappa=1/2$.
As the \Cordes\ condition for this example is fulfilled with a sufficiently large~$\constCordes$
we dropped the stabilization terms, \ie, we set $\eta_1=\eta_2=0$.
For comparison, we also present computational results for the piecewise Hessian approach (NSZ).
The error plots in different norms and for varying polynomial degrees are shown in \cref{fig:comparison_methods}.
All convergence rates in the $H_h^2(\Omega)$-norm coincide with the ones predicted by \cref{thm:a_priori_estimate}.
It is also observed that both approaches behave quite similarly. In the $H_h^2(\Omega)$-norm the errors decay almost identically. However we observe two advantages for our approach using a Hessian recovery strategy. First, it even converges in the $L^2(\Omega)$- and $H^1(\Omega)$-norm if the polynomial degree $p=1$ is used. This coincides with the observations from~\cite{LakkisPryer2013}, where the case $p=1$ is allowed as well. Second, the convergence rate in the $L^2(\Omega)$-norm is higher for the Hessian recovery approach in case of quadratic elements. This is caused by the fact that a stabilization term is not needed in the present situation.

\pgfplotstableread[col sep = comma]{./results/Degree_Study/degree_study_Cinfty_kappa_05_cg0stab_deg_2.csv}\cgZeroStabTwo%
\pgfplotstableread[col sep = comma]{./results/Degree_Study/degree_study_Cinfty_kappa_05_cg1stab_deg_2.csv}\cgOneStabTwo%
\pgfplotstableread[col sep = comma]{./results/Degree_Study/degree_study_Cinfty_kappa_05_cg2stab_deg_2.csv}\cgTwoStabTwo%
\pgfplotstableread[col sep = comma]{./results/Degree_Study/degree_study_Cinfty_kappa_05_dg0stab_deg_2.csv}\dgZeroStabTwo%
\pgfplotstableread[col sep = comma]{./results/Degree_Study/degree_study_Cinfty_kappa_05_dg1stab_deg_2.csv}\dgOneStabTwo%
\pgfplotstableread[col sep = comma]{./results/Degree_Study/degree_study_Cinfty_kappa_05_Neilan_deg_2.csv}\NeilanTwo%
\pgfplotstableread[col sep = comma]{./results/Degree_Study/degree_study_Cinfty_kappa_05_NeilanSalgadoZhang_deg_2.csv}\NeilanSalgadoZhangTwo%
\pgfplotstableread[col sep = comma]{./results/Degree_Study/degree_study_Cinfty_kappa_05_cg0stab_deg_3.csv}\cgZeroStabThree%
\pgfplotstableread[col sep = comma]{./results/Degree_Study/degree_study_Cinfty_kappa_05_cg1stab_deg_3.csv}\cgOneStabThree%
\pgfplotstableread[col sep = comma]{./results/Degree_Study/degree_study_Cinfty_kappa_05_cg2stab_deg_3.csv}\cgTwoStabThree%
\pgfplotstableread[col sep = comma]{./results/Degree_Study/degree_study_Cinfty_kappa_05_dg0stab_deg_3.csv}\dgZeroStabThree%
\pgfplotstableread[col sep = comma]{./results/Degree_Study/degree_study_Cinfty_kappa_05_dg1stab_deg_3.csv}\dgOneStabThree%
\pgfplotstableread[col sep = comma]{./results/Degree_Study/degree_study_Cinfty_kappa_05_Neilan_deg_3.csv}\NeilanThree%
\pgfplotstableread[col sep = comma]{./results/Degree_Study/degree_study_Cinfty_kappa_05_NeilanSalgadoZhang_deg_3.csv}\NeilanSalgadoZhangThree%
\begin{figure}[bht]
	\begin{subfigure}{0.33\textwidth}
		\centering
		\resizebox{\textwidth}{!}{
			\begin{tikzpicture}
				\begin{loglogaxis}[enlargelimits=false,
					xlabel={$\dim(V_h)$}]
					\addplot+[cgZeroStab] table[x=Ndofs,y=L2_error] {\cgZeroStabTwo};
					\addplot+[cgOneStab]  table[x=Ndofs,y=L2_error] {\cgOneStabTwo};
					\addplot+[cgTwoStab]  table[x=Ndofs,y=L2_error] {\cgTwoStabTwo};
					\addplot+[dgZeroStab] table[x=Ndofs,y=L2_error] {\dgZeroStabTwo};
					\addplot+[dgOneStab]  table[x=Ndofs,y=L2_error] {\dgOneStabTwo};
					\addplot+[Neilan]     table[x=Ndofs,y=L2_error] {\NeilanTwo};
					\addplot+[NeilanSalgadoZhang] table[x=Ndofs,y=L2_error] {\NeilanSalgadoZhangTwo};
				\end{loglogaxis}
			\end{tikzpicture}
		}
		\caption{$\norm{u-u_h}_{L^2(\Omega)}$ for $p=2$}
	\end{subfigure}
	\begin{subfigure}{0.33\textwidth}
		\resizebox{\textwidth}{!}{
			\centering       
			\begin{tikzpicture}
				\begin{loglogaxis}[enlargelimits=false,
						legend to name=externallegend,
						legend columns = 7,
						legend entries = {$\text{CG}_{\eta_1=0,\eta_2=0}$,
							$\text{CG}_{\eta_1=1,\eta_2=0}$,
							$\text{CG}_{\eta_1=1,\eta_2=1}$,
							$\text{DG}_{\eta_1=0}$,
							$\text{DG}_{\eta_1=1}$,
							$\text{N}$,
						$\text{NSZ}$},
					xlabel={$\dim(V_h)$}]
					\addplot+[cgZeroStab] table[x=Ndofs,y=H1_error] {\cgZeroStabTwo};
					\addplot+[cgOneStab]  table[x=Ndofs,y=H1_error] {\cgOneStabTwo};
					\addplot+[cgTwoStab]  table[x=Ndofs,y=H1_error] {\cgTwoStabTwo};
					\addplot+[dgZeroStab] table[x=Ndofs,y=H1_error] {\dgZeroStabTwo};
					\addplot+[dgOneStab]  table[x=Ndofs,y=H1_error] {\dgOneStabTwo};
					\addplot+[Neilan]     table[x=Ndofs,y=H1_error] {\NeilanTwo};
					\addplot+[NeilanSalgadoZhang] table[x=Ndofs,y=H1_error] {\NeilanSalgadoZhangTwo};
				\end{loglogaxis}
			\end{tikzpicture}
		}
		\caption{$\norm{u-u_h}_{H^1(\Omega)}$ for $p=2$}
	\end{subfigure}
	\begin{subfigure}{0.33\textwidth}
		\centering
		\resizebox{\textwidth}{!}{
			\begin{tikzpicture}
				\begin{loglogaxis}[enlargelimits=false,
					xlabel={$\dim(V_h)$}]
					\addplot+[cgZeroStab] table[x=Ndofs,y=H2h_error] {\cgZeroStabTwo};
					\addplot+[cgOneStab]  table[x=Ndofs,y=H2h_error] {\cgOneStabTwo};
					\addplot+[cgTwoStab]  table[x=Ndofs,y=H2h_error] {\cgTwoStabTwo};
					\addplot+[dgZeroStab] table[x=Ndofs,y=H2h_error] {\dgZeroStabTwo};
					\addplot+[dgOneStab]  table[x=Ndofs,y=H2h_error] {\dgOneStabTwo};
					\addplot+[Neilan]     table[x=Ndofs,y=H2h_error] {\NeilanTwo};
					\addplot+[NeilanSalgadoZhang] table[x=Ndofs,y=H2h_error] {\NeilanSalgadoZhangTwo};
				\end{loglogaxis}	  	  
			\end{tikzpicture}
		}
		\caption{$\norm{u-u_h}_{H_h^2(\Omega)}$ for $p=2$}
	\end{subfigure}
	\begin{subfigure}{0.33\textwidth}
		\centering
		\resizebox{\textwidth}{!}{
			\begin{tikzpicture}
				\begin{loglogaxis}[enlargelimits=false,
					xlabel={$\dim(V_h)$}]
					\addplot+[cgZeroStab] table[x=Ndofs,y=L2_error] {\cgZeroStabThree};
					\addplot+[cgOneStab]  table[x=Ndofs,y=L2_error] {\cgOneStabThree};
					\addplot+[cgTwoStab]  table[x=Ndofs,y=L2_error] {\cgTwoStabThree};
					\addplot+[dgZeroStab] table[x=Ndofs,y=L2_error] {\dgZeroStabThree};
					\addplot+[dgOneStab]  table[x=Ndofs,y=L2_error] {\dgOneStabThree};
					\addplot+[Neilan]     table[x=Ndofs,y=L2_error] {\NeilanThree};
					\addplot+[NeilanSalgadoZhang] table[x=Ndofs,y=L2_error] {\NeilanSalgadoZhangThree};
				\end{loglogaxis}
			\end{tikzpicture}
		}
		\caption{$\norm{u-u_h}_{L^2(\Omega)}$ for $p=3$}
	\end{subfigure}
	\begin{subfigure}{0.33\textwidth}
		\resizebox{\textwidth}{!}{
			\centering       
			\begin{tikzpicture}
				\begin{loglogaxis}[enlargelimits=false,		
					xlabel={$\dim(V_h)$}]
					\addplot+[cgZeroStab] table[x=Ndofs,y=H1_error] {\cgZeroStabThree};
					\addplot+[cgOneStab]  table[x=Ndofs,y=H1_error] {\cgOneStabThree};
					\addplot+[cgTwoStab]  table[x=Ndofs,y=H1_error] {\cgTwoStabThree};
					\addplot+[dgZeroStab] table[x=Ndofs,y=H1_error] {\dgZeroStabThree};
					\addplot+[dgOneStab]  table[x=Ndofs,y=H1_error] {\dgOneStabThree};
					\addplot+[Neilan]     table[x=Ndofs,y=H1_error] {\NeilanThree};
					\addplot+[NeilanSalgadoZhang] table[x=Ndofs,y=H1_error] {\NeilanSalgadoZhangThree};
				\end{loglogaxis}
			\end{tikzpicture}
		}
		\caption{$\norm{u-u_h}_{H^1(\Omega)}$ for $p=3$}
	\end{subfigure}
	\begin{subfigure}{0.33\textwidth}
		\centering
		\resizebox{\textwidth}{!}{
			\begin{tikzpicture}
				\begin{loglogaxis}[enlargelimits=false,
					xlabel={$\dim(V_h)$}]
					\addplot+[cgZeroStab] table[x=Ndofs,y=H2h_error] {\cgZeroStabThree};
					\addplot+[cgOneStab]  table[x=Ndofs,y=H2h_error] {\cgOneStabThree};
					\addplot+[cgTwoStab]  table[x=Ndofs,y=H2h_error] {\cgTwoStabThree};
					\addplot+[dgZeroStab] table[x=Ndofs,y=H2h_error] {\dgZeroStabThree};
					\addplot+[dgOneStab]  table[x=Ndofs,y=H2h_error] {\dgOneStabThree};
					\addplot+[Neilan]     table[x=Ndofs,y=H2h_error] {\NeilanThree};
					\addplot+[NeilanSalgadoZhang] table[x=Ndofs,y=H2h_error] {\NeilanSalgadoZhangThree};
				\end{loglogaxis}	  	  
			\end{tikzpicture}
		}
		\caption{$\norm{u-u_h}_{H_h^2(\Omega)}$ for $p=3$}
	\end{subfigure}    
	\bigskip

	\centering
	\hypersetup{hidelinks}
	\ref{externallegend} 
	\caption{This figure shows the performance of the algorithms for the example from~\cref{sec:cordes_violated} for $p=2$ and $p=3$ for different discretization strategies.}
	\label{fig:exp1_comparison_methods_deg_2}
\end{figure}
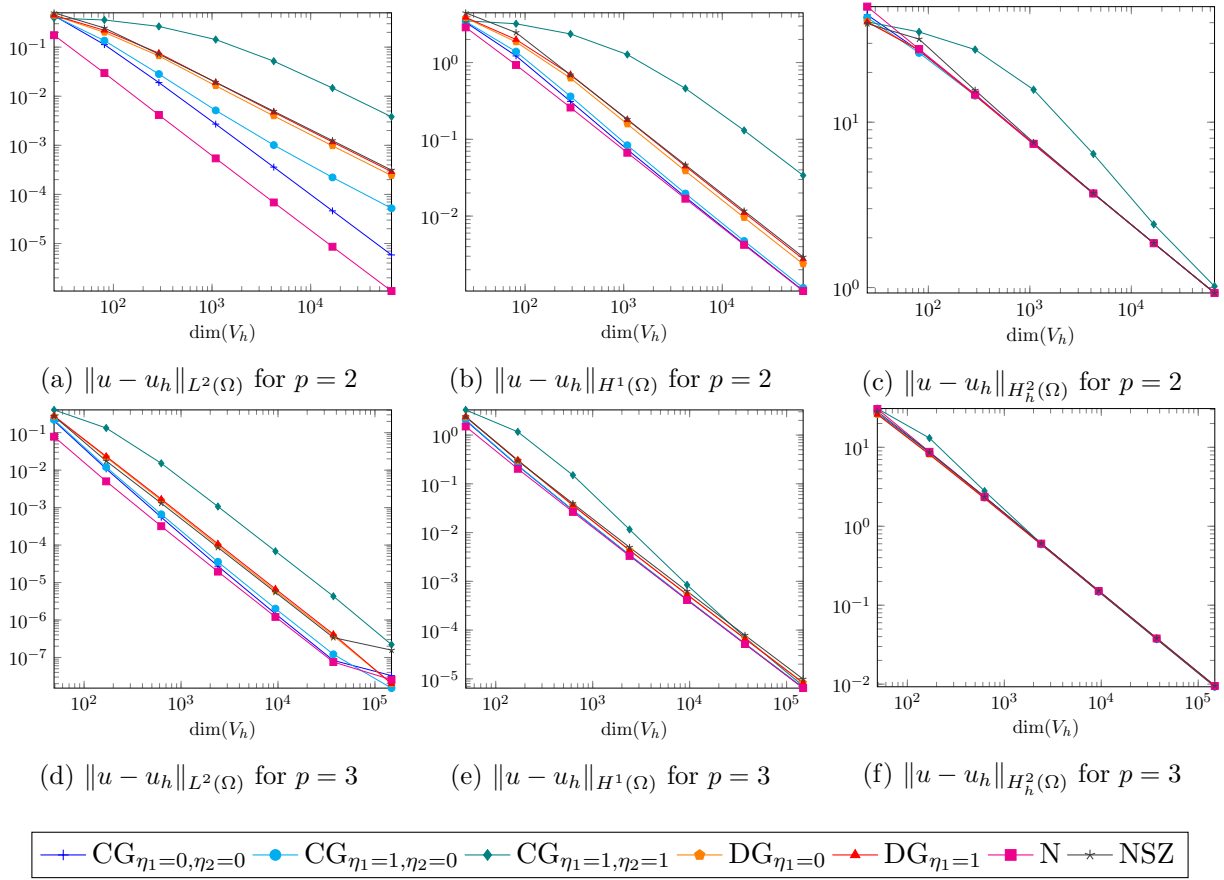

In a last test for this example we check how the methods studied in the present article compare with the approaches (N) and (NSZ)  mentioned at the beginning of this section.
The error curves for different norms and different polynomial degrees can be found in \cref{fig:exp1_comparison_methods_deg_2}. Although all approaches behave quite similarly, we observe a difference in the convergence rates in $L^2(\Omega)$ for quadratic elements. Obviously, the approaches taking into account stabilization terms (these are our approaches with $\eta_1\ne 0$ and (NSZ)) converge only with order~$2$, while the remaining approaches (these are our approach with $\eta_1=\eta_2=0$ and (N)) converge with order~$3$. A proof of this conjecture is subject of future research.

\subsection{A problem with singular solution}
\label{sec:Halpha_solution}

In this example we consider the Poisson problem, \ie, the diffusion matrix is chosen as $A = \I_{2\times2}$, in the domain $\Omega = (0,1)^2$. 
Emphasis is put on problems whose solutions have reduced regularity. To this end, we construct the right-hand side $f$ in such a way that 
\begin{equation*}
	u(x) = r(x)^{\alpha} \, \sin(2 \, \varphi(x)) \, (1 - x_1) \, (1-x_2)
\end{equation*}
is the exact solution. Here, $(r(x),\varphi(x))$ are polar coordinates centered in the origin.
A simple computation shows that $u\in H^s(\Omega)$ holds for all $s<1+\alpha$. 
In the present experiment we choose the value $\alpha=3/2$ and expect the regularity of almost $H^{5/2}(\Omega)$, and thus, as predicted by \cref{thm:a_priori_estimate}, the convergence rate in the $H_h^2(\Omega)$-norm should be $1/2-\varepsilon$ for arbitrary $\varepsilon>0$.
We would also expect that an adaptive finite element method will retain the optimal convergence
rate. The adaptive strategy we implemented uses the local
error estimator \eqref{eq:local_estimator}, the D\"orfler marking strategy in
such a way that those elements contributing 90\% to the globally estimated error are marked, and the bisection refinement strategy provided by the \fenics\ library.
The results shown in \cref{fig:adaptive_vs_uniform} confirm the optimality of the adaptively generated finite element meshes. It is also observed that the convergence rates in the $H^1(\Omega)$- and $L^2(\Omega)$-norm are optimal.

\input{experiment2_adaptive.tex}

\subsection{A problem with discontinuous coefficient matrix}
\label{sec:experiment_disc_A}

This example illustrates the capability of the method to handle discontinuous diffusion coefficients. Problems of this type are of particular interest as a transformation to a PDE in divergence form is not possible.
The coefficient matrix in the present example is 
\begin{equation*}
	A(x) = 
	\begin{pmatrix}
		2&\sgn(x_1\,x_2) \\
		\sgn(x_1\,x_2)&2
	\end{pmatrix},
\end{equation*}
and $f$ is chosen in such a way that $u(x_1,x_2) = x_1\,x_2\,(1-e^{1-\abs{x_1}})\,(1-e^{1-\abs{x_2}})$ is the exact solution. The computational domain is $\Omega:=(-1,1)^2$.
This example is also used in the numerical experiments from \cite{FengNeilanSchnake2018,SmearsSueli2014,WangWang2017}, where different discretization approaches are studied.
Here, we apply our finite element scheme from \cref{sec:discretization_approach} and investigate the behavior of an adaptive finite element method based on the error estimator derived in \cref{thm:estimator_reliable}.
The adaptively generated mesh as well as the error curves can be found in \cref{fig:FNS_adaptive_vs_uniform}.

\input{experiment3_adaptive.tex}

Finally we illustrate the convergence behavior for different choices for the polynomial degree in \cref{fig:experiment3_degree_study} and compare again our method without stabilization and the piecewise Hessian approach (NSZ). In the $H_h^2(\Omega)$-norm both approaches behave similarly and the convergence rate predicted in \cref{thm:a_priori_estimate} is also confirmed. Our approach performs even slightly better
when comparing the error in weaker norms.

\pgfplotstableread[col sep = comma]{./results/Degree_Study/degree_study_qd100_FNS_5_4_deg_1.csv}\degOne%
\pgfplotstableread[col sep = comma]{./results/Degree_Study/degree_study_qd100_FNS_5_4_deg_2.csv}\degTwo%
\pgfplotstableread[col sep = comma]{./results/Degree_Study/degree_study_qd100_FNS_5_4_deg_3.csv}\degThree%
\pgfplotstableread[col sep = comma]{./results/Degree_Study/degree_study_qd100_FNS_5_4_deg_4.csv}\degFour%
\begin{figure}[bht]
	\begin{subfigure}{0.33\textwidth}
		\centering
		\resizebox{\textwidth}{!}{
			\begin{tikzpicture}
				\begin{loglogaxis}[enlargelimits=false,
						xlabel={$\dim(V_h)$},
						legend to name=exp3degreeStudyLegend,
					legend columns = 7]
					\addplot+[cgZeroStab] table[x=Ndofs,y=cg0stab_L2_error] {\degOne};
					\addlegendentry{$\text{CG}_{\eta_1=0,\eta_2=0}$};
					\addplot+[NeilanSalgadoZhang] table[x=Ndofs,y=nsz_L2_error] {\degOne};
					\addlegendentry{$\text{NSZ}$};
					\addlegendimage{solid, black};
					\addlegendentry{$p = 1$};
					\addlegendimage{dashed, black};
					\addlegendentry{$p = 2$};
					\addlegendimage{dash dot, black};
					\addlegendentry{$p = 3$};
					\addlegendimage{densely dotted, black};
					\addlegendentry{$p = 4$};
					\addplot+[cgZeroStab, dashed] table[x=Ndofs,y=cg0stab_L2_error] {\degTwo};
					\addplot+[NeilanSalgadoZhang, dashed] table[x=Ndofs,y=nsz_L2_error] {\degTwo};
					\addplot+[cgZeroStab, dash dot] table[x=Ndofs,y=cg0stab_L2_error] {\degThree};
					\addplot+[NeilanSalgadoZhang, dash dot] table[x=Ndofs,y=nsz_L2_error] {\degThree};
					\addplot+[cgZeroStab, densely dotted] table[x=Ndofs,y=cg0stab_L2_error] {\degFour};
					\addplot+[NeilanSalgadoZhang, densely dotted] table[x=Ndofs,y=nsz_L2_error] {\degFour};
				\end{loglogaxis}
			\end{tikzpicture}
		}
		\caption{$\norm{u-u_h}_{L^2(\Omega)}$}
	\end{subfigure}\hfill
	\begin{subfigure}{0.33\textwidth}
		\centering
		\resizebox{\textwidth}{!}{
			\begin{tikzpicture}
				\begin{loglogaxis}[enlargelimits=false,
					xlabel={$\dim(V_h)$}]
					\addplot+[cgZeroStab] table[x=Ndofs,y=cg0stab_H1_error] {\degOne};
					\addplot+[NeilanSalgadoZhang] table[x=Ndofs,y=nsz_H1_error] {\degOne};
					\addplot+[cgZeroStab, dashed] table[x=Ndofs,y=cg0stab_H1_error] {\degTwo};
					\addplot+[NeilanSalgadoZhang, dashed] table[x=Ndofs,y=nsz_H1_error] {\degTwo};
					\addplot+[cgZeroStab, dash dot] table[x=Ndofs,y=cg0stab_H1_error] {\degThree};
					\addplot+[NeilanSalgadoZhang, dash dot] table[x=Ndofs,y=nsz_H1_error] {\degThree};
					\addplot+[cgZeroStab, densely dotted] table[x=Ndofs,y=cg0stab_H1_error] {\degFour};
					\addplot+[NeilanSalgadoZhang, densely dotted] table[x=Ndofs,y=nsz_H1_error] {\degFour};
				\end{loglogaxis}
			\end{tikzpicture}
		}
		\caption{$\norm{u-u_h}_{H^1(\Omega)}$}
	\end{subfigure}\hfill
	\begin{subfigure}{0.33\textwidth}
		\centering
		\resizebox{\textwidth}{!}{
			\begin{tikzpicture}
				\begin{loglogaxis}[enlargelimits=false,
					xlabel={$\dim(V_h)$}]
					\addplot+[cgZeroStab] table[x=Ndofs,y=cg0stab_H2h_error] {\degOne};
					\addplot+[NeilanSalgadoZhang] table[x=Ndofs,y=nsz_H2h_error] {\degOne};
					\addplot+[cgZeroStab, dashed] table[x=Ndofs,y=cg0stab_H2h_error] {\degTwo};
					\addplot+[NeilanSalgadoZhang, dashed] table[x=Ndofs,y=nsz_H2h_error] {\degTwo};
					\addplot+[cgZeroStab, dash dot] table[x=Ndofs,y=cg0stab_H2h_error] {\degThree};
					\addplot+[NeilanSalgadoZhang, dash dot] table[x=Ndofs,y=nsz_H2h_error] {\degThree};
					\addplot+[cgZeroStab, densely dotted] table[x=Ndofs,y=cg0stab_H2h_error] {\degFour};
					\addplot+[NeilanSalgadoZhang, densely dotted] table[x=Ndofs,y=nsz_H2h_error] {\degFour};
				\end{loglogaxis}
			\end{tikzpicture}
		}
		\caption{$\norm{u-u_h}_{H^2_h(\Omega)}$}
	\end{subfigure}
	\bigskip

	\centering
	\hypersetup{hidelinks}
	\ref{exp3degreeStudyLegend} 
	\caption{Comparison of absolute errors for different polynomial degrees $p$ for the example from~\cref{sec:experiment_disc_A} with discontinuous coefficient matrix.}
	\label{fig:experiment3_degree_study}
\end{figure}

\subsection{A problem with anisotropic and discontinuous coefficient matrix}
\label{sec:experiment_anisotropic_diffusion}

In this example we consider a problem in $\Omega\coloneqq (-1,1)^2$ with the input data
\begin{equation*}
	f=-1\quad\text{and}\quad A = \begin{pmatrix} 0.02 & 0.01 \\ 0.01 & 1 + \textbf{1}_{(x^3-y > 0)} \end{pmatrix}.
\end{equation*}
The diffusion in $x_1$-direction is very small so that the solution exhibits a boundary layer at the boundary edges $x_2=-1$ and $x_2=1$. Moreover, the coefficient $A_{22}$ is discontinuous.   
The computational results for our adaptive finite element method are illustrated in \cref{fig:exp4_adaptive_vs_uniform}.
We observe that the discontinuity and the boundary layer are both resolved by the mesh.
Furthermore, the propagation of the error is illustrated and one observes that the adaptive refinement retains the optimal convergence rate. Note that we used the value of the global estimator $\eta$ as an error measure since an explicit solution is not available for this example.

\pgfplotstableread[col sep = comma]{results/experiment4/Disc_A_deg_1_CG_1_stab_uniform.csv}\expFourCGZeroUniformOne%
\pgfplotstableread[col sep = comma]{results/experiment4/Disc_A_deg_1_CG_1_stab_adaptive.csv}\expFourCGZeroAdaptiveOne%
\pgfplotstableread[col sep = comma]{results/experiment4/Disc_A_deg_2_CG_1_stab_uniform.csv}\expFourCGZeroUniformTwo%
\pgfplotstableread[col sep = comma]{results/experiment4/Disc_A_deg_2_CG_1_stab_adaptive.csv}\expFourCGZeroAdaptiveTwo%
\pgfplotstableread[col sep = comma]{results/experiment4/Disc_A_deg_3_CG_1_stab_uniform.csv}\expFourCGZeroUniformThree%
\pgfplotstableread[col sep = comma]{results/experiment4/Disc_A_deg_3_CG_1_stab_adaptive.csv}\expFourCGZeroAdaptiveThree%
\pgfplotstableread[col sep = comma]{results/experiment4/Disc_A_deg_4_CG_1_stab_uniform.csv}\expFourCGZeroUniformFour%
\pgfplotstableread[col sep = comma]{results/experiment4/Disc_A_deg_4_CG_1_stab_adaptive.csv}\expFourCGZeroAdaptiveFour%
\begin{figure}[ht]
	\begin{subfigure}[b]{0.45\textwidth}
		\centering
		\includegraphics[width=.9\textwidth]{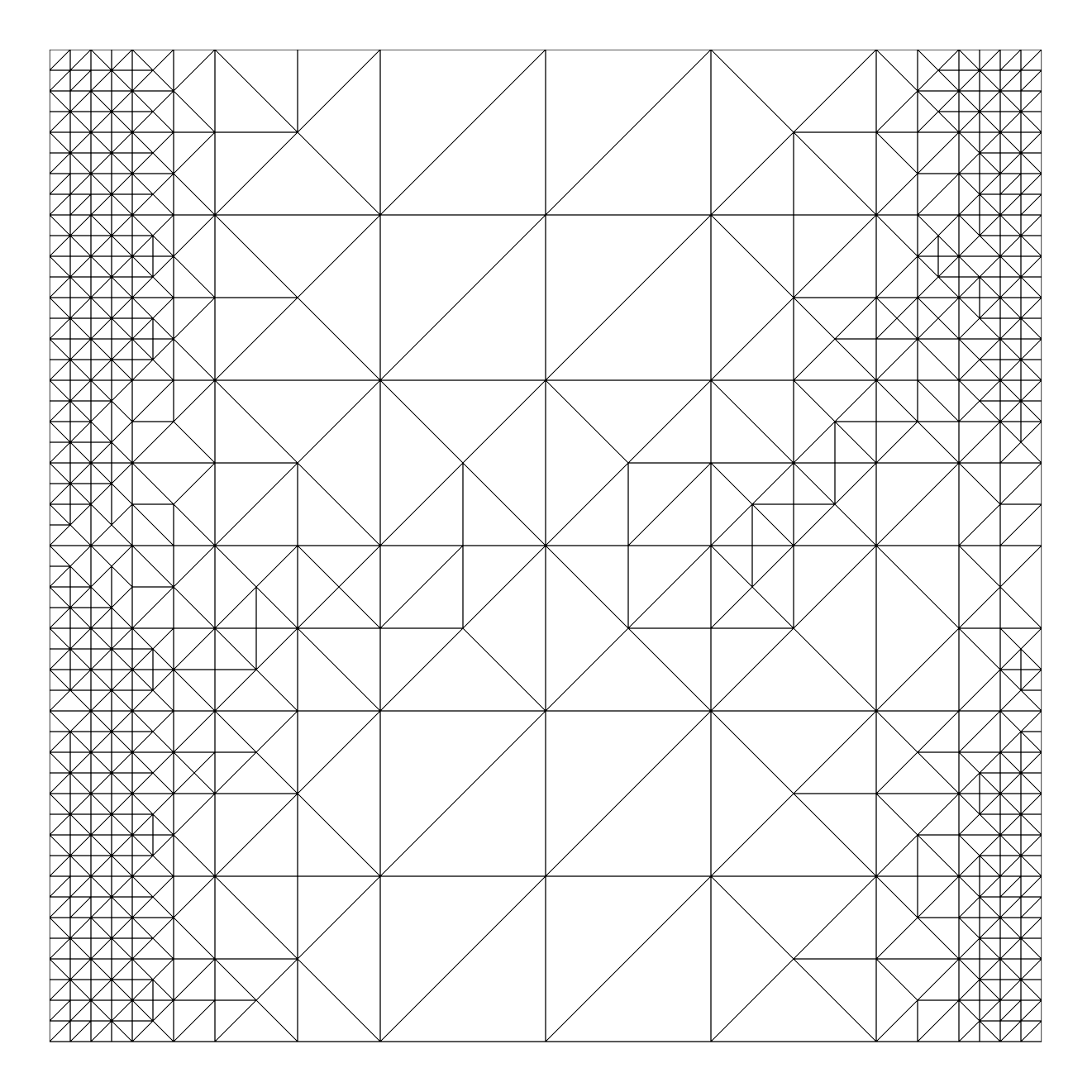}
		\caption{Adaptively generated mesh with $\num{2158}$~dofs}
	\end{subfigure}\quad
	\begin{subfigure}[b]{0.45\textwidth}
		\centering
		\includegraphics[width=.9\textwidth]{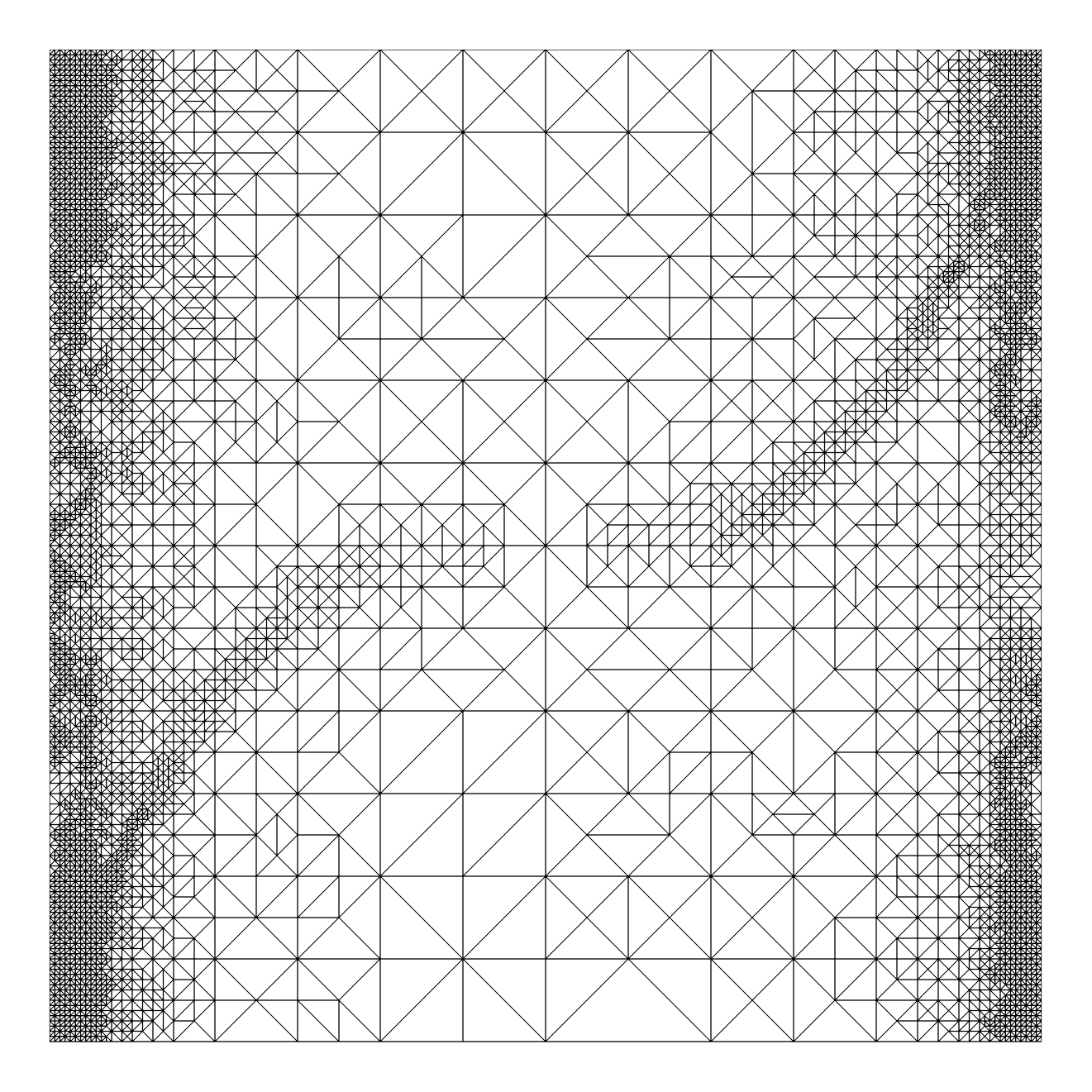}
		\caption{Adaptively generated mesh with $\num{20057}$~dofs}
	\end{subfigure}\\
	\begin{subfigure}[b]{0.45\textwidth}
		\includegraphics[width=\textwidth]{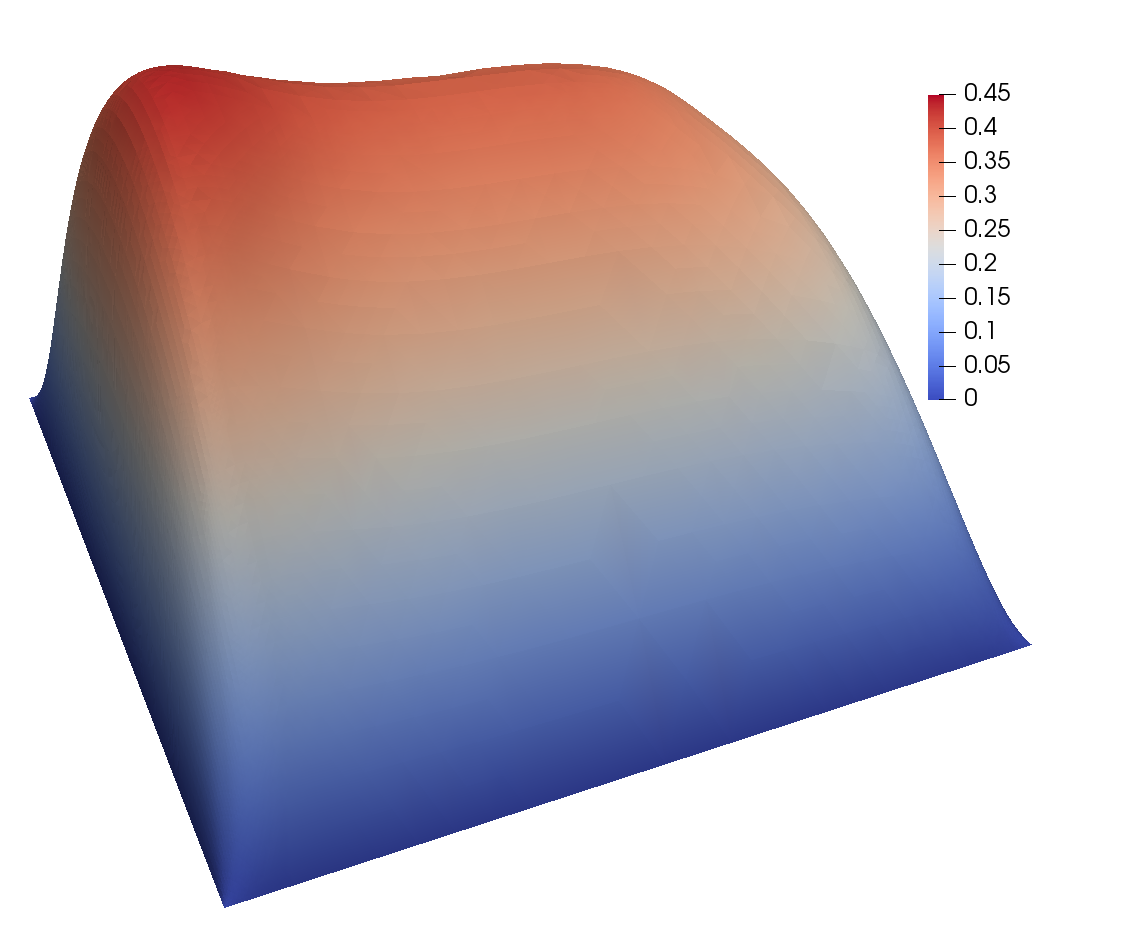}
		\caption{Numerical solution $u_h$\\\phantom{a}}
	\end{subfigure}\quad
	\begin{subfigure}[b]{0.45\textwidth}
		\centering
		\resizebox{\textwidth}{!}{
			\begin{tikzpicture}
				\begin{loglogaxis}[enlargelimits=false,
						legend pos=south west,
					xlabel={$\dim(V_h)$}]
					\addplot+[adaptiveH1] table[x=Ndofs,y=Eta_global] {\expFourCGZeroAdaptiveTwo};
					\addplot+[adaptiveH2] table[x=Ndofs,y=Eta_global] {\expFourCGZeroAdaptiveThree};
					\addplot+[adaptiveL2] table[x=Ndofs,y=Eta_global] {\expFourCGZeroAdaptiveFour};
					\addlegendentry{$p = 2$};
					\addlegendentry{$p = 3$};
					\addlegendentry{$p = 4$};
					\addplot+[uniformH1] table[x=Ndofs,y=Eta_global] {\expFourCGZeroUniformTwo};
					\addplot+[uniformH2] table[x=Ndofs,y=Eta_global] {\expFourCGZeroUniformThree};
					\addplot+[uniformL2] table[x=Ndofs,y=Eta_global] {\expFourCGZeroUniformFour};
				\end{loglogaxis}
			\end{tikzpicture}
		}
		\caption{Propagation of error estimator $\eta$}
	\end{subfigure}
	\centering
	\caption{Comparison between uniform (dashed lines) and adaptive
	refinement (solid lines) for the example from \cref{sec:experiment_anisotropic_diffusion}.}
	\label{fig:exp4_adaptive_vs_uniform}
\end{figure}

\subsection{A three-dimensional problem with reduced regularity}%
\label{sec:experiment_3d}

In this numerical experiment we show the applicability of our procedure to the three-dimensional case.
We choose the diffusion matrix to be $A = \I_{3\times3}$ in the domain $\Omega = {(0,1)}^3$.
Similar as in \cref{sec:Halpha_solution}, the solution
\begin{equation*}
    u(x) = r(x)^\alpha
\end{equation*}
with $r(x) = \sqrt{\sum_{i=1}^3{(x_i-0.5)}^2}$ possesses a reduced regularity, \ie, $u \in H^s(\Omega)$ for $s < 1.5 + \alpha$.
The results for various choices of $\alpha$ are shown in \cref{fig:adaptive_vs_uniform_3d} and confirm the expected behavior.
\pgfplotstableread[col sep = comma]{./csv/Sol_in_H_2.25_3d_uniform.csv}\ThreeDUniformLeft%
\pgfplotstableread[col sep = comma]{./csv/Sol_in_H_2.25_3d_adaptive.csv}\ThreeDAdaptiveLeft%
\pgfplotstableread[col sep = comma]{./csv/Sol_in_H_2.50_3d_uniform.csv}\ThreeDUniformRight%
\pgfplotstableread[col sep = comma]{./csv/Sol_in_H_2.50_3d_adaptive.csv}\ThreeDAdaptiveRight%
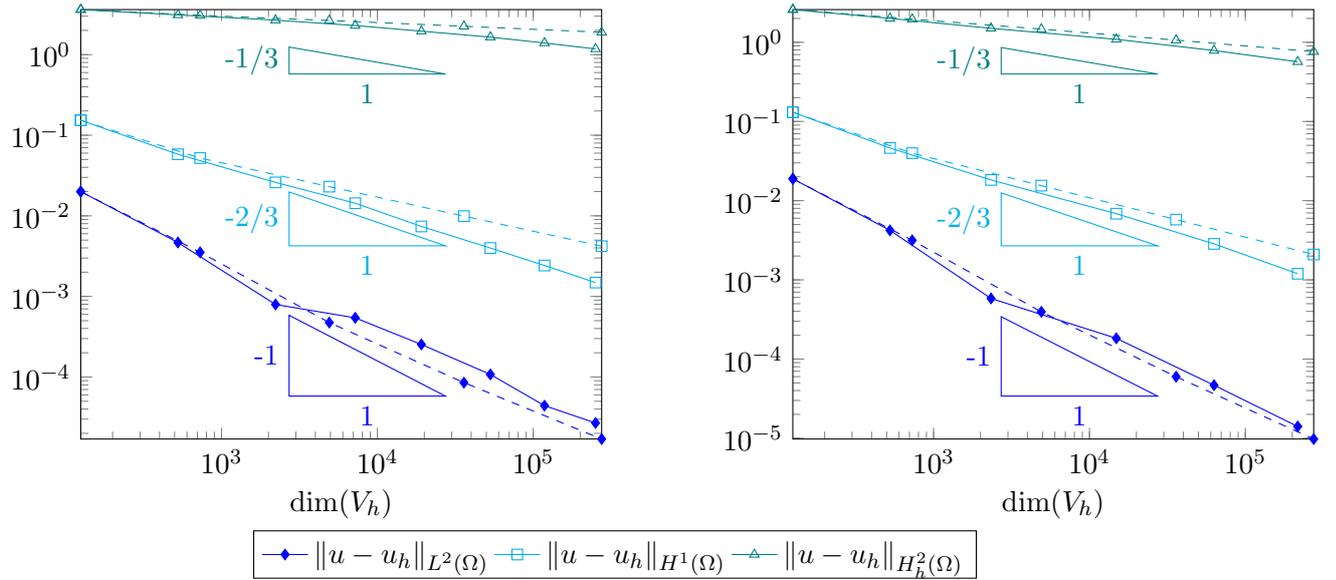
\begin{figure}[ht]
    \begin{subfigure}[b]{0.42\textwidth}
		\centering
			\begin{tikzpicture}
				\begin{loglogaxis}[enlargelimits=false,
						legend to name=exp6legend61,
						legend columns = 7,
					xlabel={$\dim(V_h)$}]
					\addplot+[adaptiveL2] table[x=Ndofs,y=L2_error] {\ThreeDAdaptiveLeft};
					\addplot+[adaptiveH1] table[x=Ndofs,y=H1_error] {\ThreeDAdaptiveLeft};
					\addplot+[adaptiveH2] table[x=Ndofs,y=H2h_error] {\ThreeDAdaptiveLeft};
					\addplot+[uniformL2] table[x=Ndofs,y=L2_error] {\ThreeDUniformLeft};
					\addplot+[uniformH1] table[x=Ndofs,y=H1_error] {\ThreeDUniformLeft};      
					\addplot+[uniformH2] table[x=Ndofs,y=H2h_error] {\ThreeDUniformLeft};
					\addlegendentry{$\norm{u-u_h}_{L^2(\Omega)}$};
					\addlegendentry{$\norm{u-u_h}_{H^1(\Omega)}$};
					\addlegendentry{$\norm{u-u_h}_{H_h^2(\Omega)}$};
					\logLogSlopeTriangle{0.4}{-0.3}{0.1}{-1}{blue};
					\logLogSlopeTriangle{0.4}{-0.3}{0.45}{-2/3}{cyan};
					\logLogSlopeTriangle{0.4}{-0.3}{0.85}{-1/3}{teal};	      
				\end{loglogaxis}
			\end{tikzpicture}
        \end{subfigure}\hfill%
    \begin{subfigure}[b]{0.42\textwidth}
		\centering
			\begin{tikzpicture}
				\begin{loglogaxis}[enlargelimits=false,
						legend to name=exp6legend62,
						legend columns = 7,
					xlabel={$\dim(V_h)$}]
					\addplot+[adaptiveL2] table[x=Ndofs,y=L2_error] {\ThreeDAdaptiveRight};
					\addplot+[adaptiveH1] table[x=Ndofs,y=H1_error] {\ThreeDAdaptiveRight};
					\addplot+[adaptiveH2] table[x=Ndofs,y=H2h_error] {\ThreeDAdaptiveRight};
					\addplot+[uniformL2] table[x=Ndofs,y=L2_error] {\ThreeDUniformRight};
					\addplot+[uniformH1] table[x=Ndofs,y=H1_error] {\ThreeDUniformRight};      
					\addplot+[uniformH2] table[x=Ndofs,y=H2h_error] {\ThreeDUniformRight};
					\addlegendentry{$\norm{u-u_h}_{L^2(\Omega)}$};
					\addlegendentry{$\norm{u-u_h}_{H^1(\Omega)}$};
					\addlegendentry{$\norm{u-u_h}_{H_h^2(\Omega)}$};
					\logLogSlopeTriangle{0.4}{-0.3}{0.1}{-1}{blue};
					\logLogSlopeTriangle{0.4}{-0.3}{0.45}{-2/3}{cyan};
					\logLogSlopeTriangle{0.4}{-0.3}{0.85}{-1/3}{teal};	      
				\end{loglogaxis}
			\end{tikzpicture}
        \end{subfigure}
	\centering
	\hypersetup{hidelinks}
	\ref{exp6legend61}
        \caption{Comparison between uniform (dashed lines) and adaptive refinement (solid lines) for the example from \cref{sec:experiment_3d}, left $\alpha = 2.25$, right $\alpha = 2.5$.
    }%
	\label{fig:adaptive_vs_uniform_3d}
\end{figure}
For example, the selection $\alpha = 1$ results in a regularity of almost $H^{5/2}(\Omega)$, which in turn yields an expected convergence rate in the $H_h^2(\Omega)$-norm of $1/2-\varepsilon$ for arbitrary $\varepsilon>0$ and $3/2 - \varepsilon$ and $5/2 - \varepsilon$ in the $H^1(\Omega)$- and $L^2(\Omega)$-norms, respectively.
The corresponding error plor in~\cref{fig:adaptive_vs_uniform_3d} confirms this.
An adaptive refinement strategy using the local error estimator~\eqref{eq:local_estimator} with a refinement threshold of 95\% is capable of confirming the convergence rate of the errors in the $L^2(\Omega)$-norm and improving the convergence rates in the $H^1(\Omega)$ and $H^2_h(\Omega)$-norm.

\section{Conclusion and outlook}%
\label{section:conclusion}

The proposed method can be extended to parabolic problems.
Given a regular solution and an appropriate time stepping scheme one can observe the same convergence rates as in the elliptic case.
Numerical tests have been performed to confirm this and they are included in the GitHub repository accompanying the paper.
A detailed analysis as in~\cite{SmearsSueli2015} is left to future research.
Another subject, and this was the authors' original motivation to study this topic, is the application of the proposed discretization to Hamilton-Jacobi-Bellman equations. A preliminary implementation is also available in the repository and the related theoretical foundation will be examined in further publications.

Besides these two extensions there are further interesting questions left. An obvious question is the proof for error estimates in lower-order norms. Note that one advantage of the approach proposed in this article is that optimal convergence in $L^2(\Omega)$ is observed. However, a proof of this observation is still missing. To the best of the authors' knowledge the only article dealing with estimates in lower-order norms is~\cite{FengSchnake2019_preprint}, where an $H^1(\Omega)$-norm estimate for the Petrov-Galerkin approach using $\tau_h=\text{id}$ is shown. Related studies for methods exploiting the Cordes condition are not available in the literature. A proof based on the usual duality argument is likely not expedient as, for instance, the dual equation of a non-divergence form PDE is a PDE in double-divergence form whose solutions possess insufficient regularity. In the special case $A=\text{id}$ and for the method proposed in \cref{sec:PiecewiseHessianApproach}, the discretization coincides with a $\mathcal{C}^0$-interior penalty discretization of the biharmonic equation and error estimates in lower-order norms can be directly concluded from~\cite{BrennerSung2005}. However, an extension to approaches using recovered Hessians is not straightforward and requires further investigations. 

\section*{Acknowledgments}

The authors thank Martin Stoll for discussions on preconditioners for our equation system.
\bigskip

This work was partially supported by DFG grant HE~6077/7--1 \emph{Impulse Control Problems and Adaptive Numerical Solution of Quasi-Variational Inequalities in Markovian Factor Models}.
Funding is gratefully acknowledged.

\printbibliography[sorting=nyt]

\end{document}